\documentclass{article}
\usepackage{amsmath,amsthm,amssymb,amsfonts}
\usepackage{lineno,graphicx,enumitem,color}
\usepackage[authoryear,round,comma]{natbib}

\newtheorem{theorem}{Theorem}[section]

\newtheorem{lemma}[theorem]{Lemma}

\title{A PDE-ODE Coupled Spatio-Temporal \\ Mathematical Model for Fire Blight During Bloom\\}
\author{Michael Pupulin$^1$, 
Xiang-Sheng Wang$^2$, \\
 Messoud A Efendiev$^3$, Thomas Giletti$^4$, Hermann J. Eberl$^1$ \\ \\
$^1$ Dept. Mathemtics and Statistics, University of Guelph, \\Guelph ON, Canada \\ \\ 
$^2$ Dept. Mathematics, University of Louisiana at Lafayette,\\ Lafayette, LA, USA\\ \\
$^3$ Dept. Mathematics, Marmara University, Istanbul, Turkey\\
and   Inst. Comp. Biology, Helmholtz Center Munich, Germany\\ \\
$^5$ Laboratoire de Math\'{e}matiques Blaise Pascal,\\ University Clermont-Auvergne, France }
\begin{document}
\maketitle

\begin{abstract}
 Fire blight is a bacterial plant disease that affects apple and pear trees.
 We present a mathematical model for its spread in an orchard during bloom. This is a PDE-ODE coupled system, consisting of two semilinear PDEs for the pathogen, coupled to a system of three ODEs for the stationary hosts. Exploratory numerical simulations suggest the existence of travelling waves, which we subsequently prove, under some conditions on parameters, using the method of upper and lower bounds and Schauder's fix point theorem. Our results are likely not optimal in the sense that our constraints on parameters, which can be interpreted biologically,  are sufficient for the existence of travelling waves, but probably not necessary.  Possible implications for fire blight biology and management are discussed.

{\bf keywords:} {\it Erwinia amylovora}; Fire blight; Mathematical Model; PDE-ODE coupled system; Travelling Waves;

\end{abstract}

\section{Introduction}

Fire blight is a bacterial plant disease. It is caused by \textit{Erwinia amylovora} and infects members of the Rosaceae family such as apple and pear trees. One characteristic symptom of fire blight is that it causes infected tissue to appear as though it had been scorched by fire, from which the disease's common name derives. The disease is persistent, incurable and highly destructive, capable of killing young apple and pear trees within a single growing season. Recent fire blight epidemics are estimated to have caused financial losses on the scale of tens of millions of dollars \citep{van2012losses}.

\subsection{Historical overview} 
The scientific pursuit of discovering the causal agent of fire blight began well before it was understood that bacteria could in fact cause diseases. The first known report of fire blight occurred over 200 years ago, when William Denning described the ``disorder" of apple and pear trees near the Hudson River Valley, New York \citep{denning1794decay}. After finding two worms in an infected host, he was the first to hypothesize the ``Insect Theory" -- that the disease was induced by insects. This theory remained as the most popular explanation for almost one hundred years, however, not everyone accepted it as satisfactory, as infected trees were observed to secrete some sort of ooze, see Figure \ref{oozepic}, and it was unclear how an insect could cause such a symptom. 

In 1878, almost one hundred years after Denning first writes about fire blight, Professor Thomas Burrill of the University of Illinois identifies the existence of bacteria within fire-blight-diseased tissue. He shows, for the first time, that bacteria could cause diseases in plants, correctly identifying the causal agent of fire blight to be \textit{Erwinia amylovora}, though it was known to him at the time as \textit{Micrococcus amylovorus}. In 1880, Burrill presents his work on fire blight, referring to it as ``The Anthrax of fruit trees" \citep{burrill1880anthrax}, publishing the first pictures of the bacterium in 1881.

Starting in the second half of the twentieth century, quantitative models of fire blight infection began to be developed. The first was the New York system in 1955, in which the correlation between temperature above a certain threshold and fire blight infections was studied by Mills \citep{mills1955fire}. He found a significant positive correlation between his temperature-based model, together with precipitation during bloom, and fire blight infections in Lake Ontario County from 1918 to 1954 \citep{van2012losses}. A good number of fire blight prediction models stemmed from this idea throughout the 1960s, 1970s, 1980s and 1990s. Perhaps most notably were the Billing systems \citep{billing1979warning}, the Cougar blight model \citep{smith1992predictive} and  the Maryblyt model \citep{steiner1989predicting}. Although these models have undoubtedly helped growers limit their losses due to fire blight, epidemics continue to occur, and the effect of the most valuable control agent against fire blight, the antibiotic streptomycin, continues to decline in the face of the  rise of antimicrobial resistance \citep{jones2000development}. 

Today, fire blight epidemics remain as a major threat to apple and pear growing operations. The disease can now be found across North America, Europe, the Mediterranean region and the Pacific Rim. For a full list of countries in which fire blight has been reported, see \citet{van2012losses}. The past few decades have seen increasingly more severe incidences of fire blight in the major pome-fruit producing states of Washington and Michigan, causing losses on the scale of tens of millions of dollars in the late 1990s \citep{van2012losses}. Still, there exists no cure for fire blight. Surprisingly, it is the suggestion by Lowell in 1826, that the most satisfactory control measure is the removal of limbs thirty centimeters below the visibly infected tissue, that remains the most effective strategy in the control of fire blight apart from the application of antibiotics during bloom.

\begin{figure}
    \centering
    \includegraphics[width=0.6\textwidth]{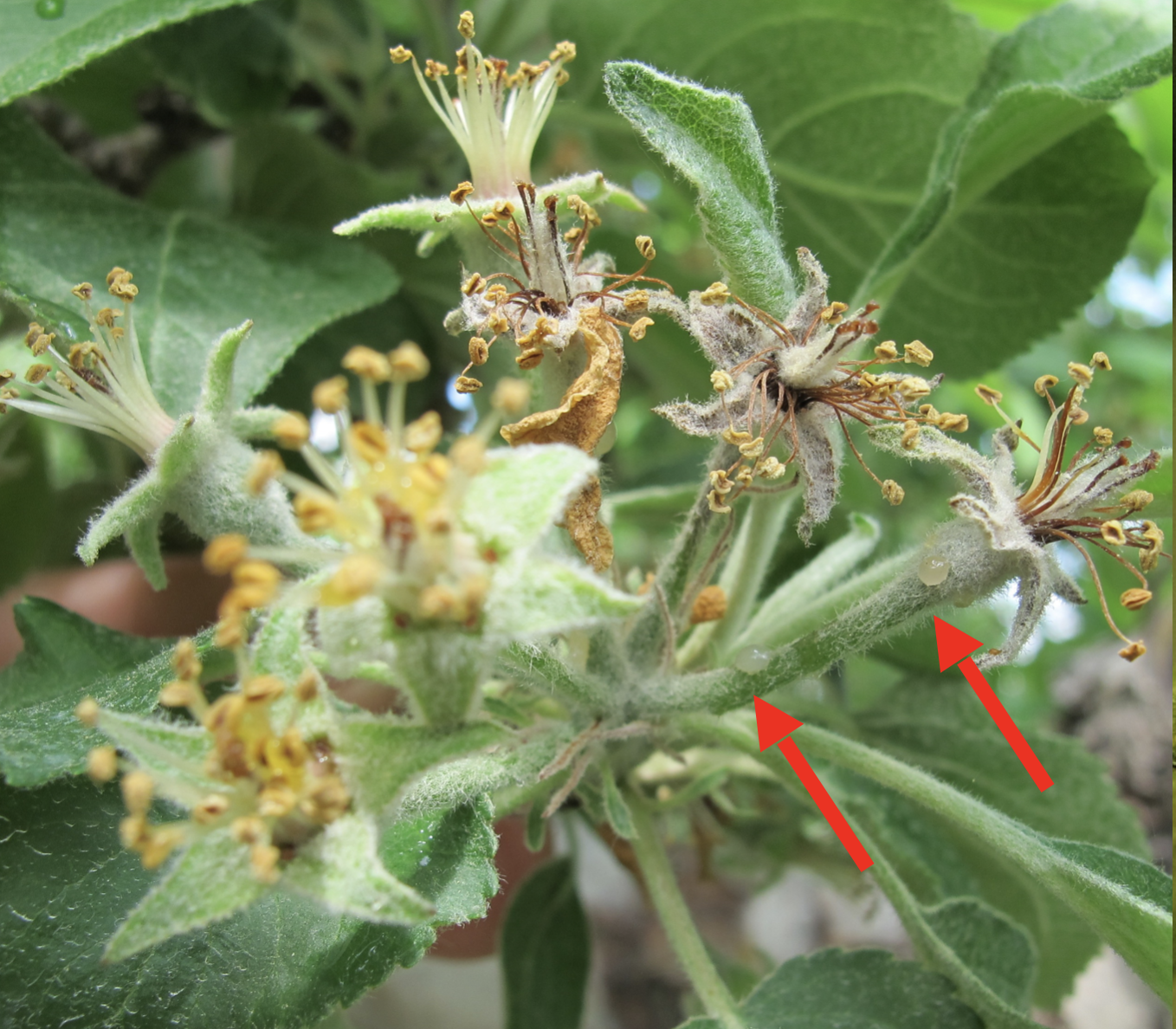}
    \caption{An infected cluster of blossoms exuding ooze. Reproduced with permission from Suzanne Slack and originally published in \citet{slack2017news}.}
    \label{oozepic}
\end{figure}

\subsection{The causal agent and the disease cycle} 
\textit{Erwinia amylovora}, the causal agent of fire blight, is a rod-shaped, gram-negative and motile species of bacteria. Although fire blight was originally reported to be a disease of apples and pears, the range of hosts which this bacterium can invade has since grown, and is summarized thoroughly in \citet{van2012losses}. It is important to note, for those interested in early reports on the bacterium, that the name of this pathogen underwent a series of changes between 1880 and 1920, before eventually falling under the \textit{Erwinia} class named after Erwin F. Smith. In order, these changes are: \textit{Micrococcus amylovorus} (1882), \textit{Bacillus amylovorus} (1889), \textit{Bacterium amylovorus} (1897), \textit{Bacterium amylovorum} (1915) and \textit{Erwinia amylovora} (1920).

Populations of \textit{E.amylovora} possess the ability to overwinter within infected trees through the formation of cankers \citep{rosen1933further,hildebrand1936overwintering}, allowing the pathogen, after an infection has occurred, to continue invading the tree until it is dead. Infected host tissue, such as these cankers, can release massive pathogen populations that are encased in an exopolysaccharide matrix \citep{eden1974bacterial}, commonly referred to as ooze. This ooze can hold bacterial populations upwards of one billion cells \citep{slack2017news}, range in size and colour \citep{slack2017microbiological}, and can serve as a source of infection for over a year \citep{hildebrand1939studies}. 

For previously blighted orchards, many believe that ooze acts as the primary source of infection in the new growing season, as over-wintered cankers can exude ooze in the spring. However, this may still be controversial, as it is noted that infections sometimes appear before ooze is detected \citep{van2012losses}. Once ooze emerges from a tree it can be disseminated by insects or rain. Flies, for example, appear to be attracted to ooze and it has been demonstrated that they can acquire and transport this ooze under various conditions \citep{boucher2019effects}. When ooze is delivered to open blossoms, the large pathogen population can be rapidly dispersed by pollinating bees and an epidemic can occur. For orchards that are infection-free, the pathogen can be introduced through a variety of ways, such as rain, insects, wind, birds and people. 

\begin{figure}
    \centering
    \includegraphics[width=0.6\textwidth]{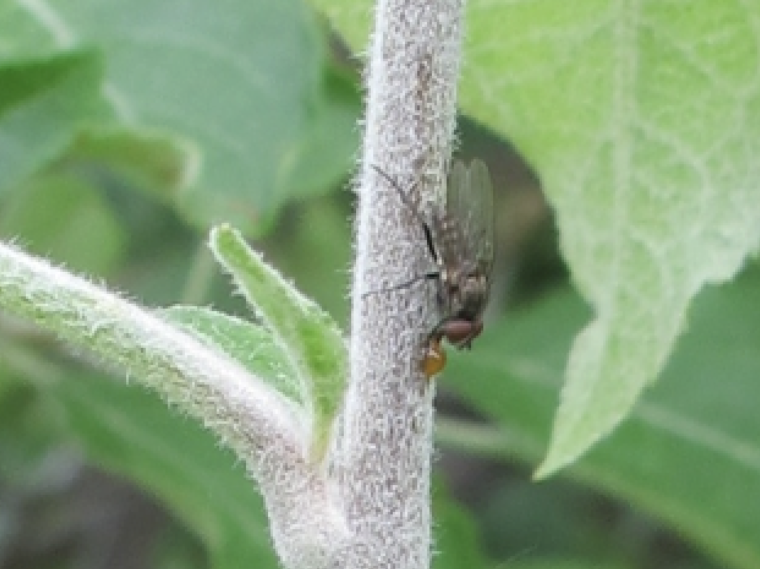}
    \caption{A fly coming into contact with fire blight ooze. Reproduced with permission from Suzanne Slack and originally published in \citet{slack2017news}.}
    \label{fig:my_label2}
\end{figure}

There are many stages of fire blight infection, each of which is often referred to by a unique name. For example, blossom blight would refer to an infection of an open blossom by the fire blight pathogen. Shoot blight, trunk blight and root blight are examples of other names that correspond to a fire blight infection of the shoot, trunk or root, respectively. Such infections can occur through natural openings in the tree, or through entry points caused by injury. As there is no cure for fire blight, most growers are concerned with preventing infection during the time of year in which trees are most susceptible, which is the blooming season. It is for this reason that in this work we focus solely on the infection of blossoms during bloom and the transmission of the disease between them. 

When the blossoms of apple and pear trees open in the spring, they provide a natural point of entry for the bacteria to invade the tree \citep{miller1929studies,rosen1929life}. Free moisture from rain or dew is required for the bacteria to move from places like the stigma to the hypanthium where infection can occur \citep{thomson1986role}. In \citet{rosen1936mode}, the author not only shows that the nectarial surface is the most susceptible location on the tree to infection, but also that the nectar serves as a good medium for the reproduction of the pathogen. 

The growth of the pathogen on floral surfaces appears to follow something similar to a sigmoidal curve \citep{wilson1989erwinia, wilson1990erwinia}. The bacterial population seems to grow exponentially when placed upon a young flower before slowing down and reaching a maximum population size. This may be due to intraspecific competition in the \textit{E.amylovora} population or due to the flower naturally becoming a less habitable place for the bacteria over time \citep{slack2022orchard, thomson2003influence}.

Transmission of the bacteria from blossom to blossom can occur via insects, rain, wind and human. Waite  was the first to notice that bees could transport the bacteria between flowers during pollination \citep{waite1891results}, and this has been confirmed since then \citep{van1980influence}. It seems safe to assume that these pollinators disseminate the bacteria faster than any other mode of dispersal. However, general opinion appears to be that bees are less responsible for fire blight epidemics than previously believed in the twentieth century, primarily due to their lack of observed contact with ooze. 

\begin{figure}
    \centering
    \includegraphics[width=\textwidth, height=5in]{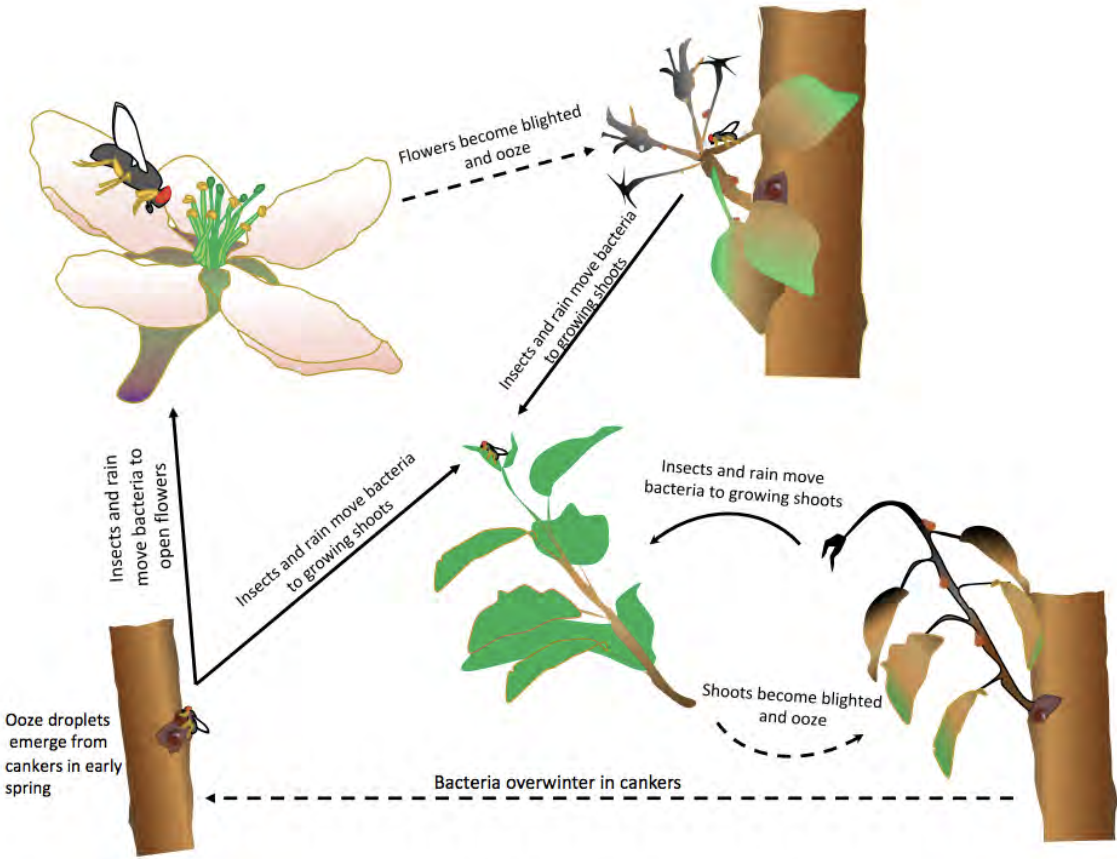}
    \caption{The fire blight disease cycle. Reproduced with permission from Suzanne Slack and originally published in \citet{slack2017news}.}
    \label{fig:my_label1}
\end{figure}

\subsection{Dynamic fire blight disease models}

To the best of our knowledge, there exist two dynamic disease models of fire blight. By a dynamic disease model, we mean a system of differential equations with the explicit goal of studying changes in disease severity through the compartmentalization of the host population into the usual classifications of susceptible, infected or removed, as done in the foundational work of Kermack and McKendrick \citep{kermack1927contribution}.  We do not review here, and comment on, the quantitative models that we mentioned earlier, such as \citet{billing1979warning,smith1992predictive,steiner1989predicting}; these models use statistical techniques to look for correlations between environmental conditions and real life infection data. While these tools  are undoubtedly valuable to growers, the objective of their models are different than the one presented in this work. Here, we aim to understand how different biological mechanisms and ecological interactions might alter, induce or prevent a fire blight epidemic. 

The earliest dynamic disease model to describe an outbreak of fire blight that we came across is \citet{iljon2012mathematical}. The authors formulated a model of six ordinary differential equations to study the daily changes in disease severity. The first four equations represent the host population, trees, classified into the four categories of susceptible and sprayed (with pesticide), susceptible and not sprayed, infected and sprayed, and infected and not sprayed. The other two equations divide pollinators into two classes, pollinators that carry the causal agent and those that do not. The authors explicitly assume that, since pollinators search for food randomly, susceptible hosts that neighbour an infected host are just as likely to become infected as susceptible hosts on the other side of the orchard. 

This model of six equations is then split into two simpler models for analysis. In the first model, the amount of antibiotic spray on each tree is held at a constant value for all time, allowing the authors to work with a model of four equations. The authors motivate the use of this constant-spray model by arguing that one could average out the amount of antibiotic applied over an entire year to a constant amount per day. They go on to study the partial rank correlation coefficients between their model parameters and their formulation of the basic reproductive number $R_0$. They find that both the size of the orchard and the number of bees within the orchard increase $R_0$, while spray efficacy and the rate at which infected branches are removed reduce $R_0$. With their partial rank correlation coefficients, they find that ``the impact of infection through nature appears insignificant when compared to infection from pollinators". 

The second simplification of the full model is realized through assuming that there is a constant number of bees carrying the pathogen per tree, therefore removing the two equations describing the transition of pollinators between the carrying-bacteria and not-carrying-bacteria classes. Under this model, they again find through a partial rank correlation study that orchard (host population) size is positively correlated with their formulation of $R_0$. They find that the rate at which one prunes infected branches is more negatively correlated with $R_0$ than spray efficacy, though neither of their partial rank correlation coefficients are necessarily strongly negatively correlated with $R_0$. 


In the second dynamic disease model of fire blight \citep{chen2018sliding}, the authors put an emphasis on the economics of pruning infected branches and replacing infected trees with susceptible ones. They formulate a Filippov model in which there are infection thresholds, such that control actions are only taken when disease levels surpass this value. This model takes the form of two ordinary differential equations, in which hosts are classified as susceptible or infected, with discontinuous right hand sides as the value of some model parameters change once an infection threshold has been passed. They make, as in the first model, the assumption that susceptible hosts who neighbour infected trees are just as likely to become infected as susceptible hosts on the other end of the orchard.  

They find that, between the control measures of pruning infected branches and replanting suggestible trees, that ``cutting off infected branches plays a leading role in reducing fire-blight infection, while the strategy of replanting susceptible trees contributes to minimizing economic losses and maximizing fruit production." They also find that the choices of the infection thresholds are important, and that the initial conditions of the disease can significantly impact what the appropriate control strategy is.

\subsection{Plant disease models and travelling waves}

Differential equation models are a common approach to understanding changes in disease severity in plant populations \citep{madden2007study}, and to aid in the development and assessment of remedial strategies. Typically, they aim to understand how the interactions between host, pathogen and the environment induce, prevent or alter an epidemic. 

Although plant disease models are more frequently studied in a purely temporal domain, implicitly assuming that the disease is homogeneous across the orchard or field, some works have focused on spatio-temporal epidemics. In some cases \citep{minogue1983models,van1977focus}, these models predict that the pathogen invades the host population with a constant speed of spread. There exists some experimental results to confirm these predictions \citep{minogue1983models}.
Mathematically, these are often described by travelling waves.  Models that admit travelling wave solutions arise in ecological, disease and population modelling \citep{diekmann1978thresholds, fisher1937wave, okubo2001diffusion}. 

Proving that a model admits travelling wave solutions can be a formidable task; for a more general overview of travelling wave solutions, see \citet{volpert1994traveling}. The existence of such solutions is often proved through arguments studying the stability of equilibrium points, taking a shooting-method approach or through the use of functional analysis tools. Although these methods have been applied to a variety of problems with success, e.g. \citet{dunbar1983travelling,logan2008introduction,volpert1994traveling}, applying them to specific systems of several equations is often challenging. A fixed point theorem based method, originally developed to prove the existence of travelling waves in delay-reaction-diffusion equations \citep{ma2001traveling}, has been successfully applied to a variety of models in recent years \citep{bhkr2005,shu2019traveling,wang2016traveling,wang2012traveling}. 

When modelling the invasion of a mobile pathogen into a stationary host population (such as in the case of plant diseases), one obtains a system of ordinary differential equations, coupled in each point of the domain to a system of partial differential equations of reaction-diffusion type. These are then parabolic systems that are everywhere degenerate, which poses additional difficulties for the analytical treatment.

Nevertheless, there is a considerable amount of work on studying travelling waves of such PDE-ODE coupled systems, in various application areas  \citep{Britton1991,DucrotGiletti,dunbar1983travelling,Hosono1995,Kallen1984,Logan2001hydrogeo}. Often this utilises a model specific relationship between the reaction terms. These approaches fail in the absence of such special structures or in the case of systems of more than two or three equations, which is the case in the model that we present. Here the ODE system is essentially an SIR type infectious disease model, in which infection rates depend on the pathogen population, which  is governed by a system of two semilinear diffusion-reaction equations. In our case the reaction terms of the disease model do not appear in the population model, and {\it vice versa}, the growth terms of the PDE component do not appear in the disease model. To study the possibility of travelling waves in this system we aim for a more general approach and adapt the above mentioned fixed point approach and extend the construction of the integral map on which it relies to such PDE-ODE coupled systems.

\section{Mathematical model} \label{chapter:banding}

We formulate a mathematical model for the spread of fire blight. We focus here on the time of bloom, and give up the restriction of previous models that the disease spread in an orchard is spatially homogeneous.

\subsection{Model assumptions}\label{ModAs:sec}

We list here the assumptions on which our model formulation will be based:

\begin{enumerate}[label=(A\arabic*)]

\item\label{A1}{\bf A spatial model is necessary.} Our first assumption is that fire blight is best modelled by a spatio-temporal system of equations. In contrast to the previous models of fire blight, we do not accept the idea that an infected tree is just as likely to transmit the disease to a neighbouring host as it is to a tree on the other end of the orchard. The other models justified this approach by saying that pollinators search for food randomly, and so transport the pathogen randomly. However, recorded measurements of bumble bee flight patterns \citep{heinrich1979resource} indicate that their movements tend to follow a long-tailed distribution, which is perhaps closer to the form of a L\'evy distribution \citep{viswanathan1999optimizing}. Such a distribution suggests that bumble bees tend to make shorter flights in a random direction, but are capable of occasionally flying longer distances. Even if one does not accept a L\'evy distribution for the flight pattern, one can recognize that the longest recorded flight lengths of bumble bees in \citep{heinrich1979resource} are orders of magnitudes smaller than a typical orchard size. If we assume that all orchard pollinators behave this way, and that the orchard is sufficiently large in comparison to the amount of land covered by a pollinator per day, then one can argue that a spatial model is necessary. 

Spatial consideration is also important from the viewpoint of disease management. In \citet{sanchirico2005optimal}, it is shown that neglecting spatial effects can lead to sub-optimal control strategies. Current fire blight control strategies often take the form of a homogeneous application of the pesticide Streptomycin. Although Streptomycin remains an indispensable tool, getting a uniform coverage of blossoms in an orchard with high foliage can be difficult. Further, the application of such pesticides can lead to an increase in antibiotic-resistant strains of the pathogen, as seen in \citet{el1989distribution,thomson1992presence}. Integrated pest management (IPM) is an approach to pest control that aims to minimize economic, health and environmental risk. Some examples of IPM strategies in the case of fire blight would be the removal of fire blight infected tissue, the bio-control of certain insects through the introduction of natural predators or trapping, or the targeted delivery of pathogen control substances through pollinators. Insight into how to optimize IPM strategies may be gained by studying how vector, host and pathogen interact in a spatial setting. It is in part the objective of this work to contribute to the discovery of new methods for the control of fire blight, by contributing to a better understanding of how the disease spreads. 

Finally, we note that our (M.P.'s) personal observation (without quantifying it) from field work is that typically fire blight exists in some particular area of the orchard, often along a single row of trees or in a corner, and that it is not randomly occurring throughout the orchard, as one might expect if each tree is equally likely to come in contact with the pathogen.  This too serves as a source of motivation for considering a spatial model of fire blight spread. 

\item\label{A2}{\bf Compartmentalization.}  We assume that we can characterize the bacteria in the orchard into two states. The first is the bacterial population existing on a susceptible location within the floral cup, and the second is the bacterial population living within the secreted ooze outside of the floral cup. With regards to the ooze, we are assuming that as long as it is present, it can be transported and used to infect hosts. Implicitly,  we are thus assuming that the ooze does not dry up and become inaccessible for insects to spread. 

We also assume that the hosts of this disease model, open blossoms, can be categorized as either susceptible to the disease, infected and infectious through the production of ooze, or removed (meaning dead). 

\item\label{A3}{\bf Homogeneous host distribution.} We make the assumption that the host population is initially distributed evenly across the spatial domain. One could argue that the more accurate representation of the host population is one that is patchy in space, with clusters of the host existing at random points throughout the domain. However, we assume here that blossoms are dense enough in an orchard such that we take an average number of blossoms at each point in space without losing biological relevancy. This might be a reasonable assumption that approximates the situation in many industrial orchards well.

\item\label{A4}{\bf Pathogen multiplication on flowers.}  The growth of the pathogen within the floral cup is dependent on the current size of the bacterial population and the amount of resources available for consumption. We assume that the amount of resources available for consumption is dependent on the health of the flower, such that as the flower begins to die, the amount of resources available for consumption decline. We further assume that when all flowers are dead at a particular location, there can still exist some small, possibly dormant, population of the pathogen. 

\item\label{A5}{\bf Ooze production.} We assume that the production and secretion of ooze is a function of the number of flowers currently infected. 

\item\label{A6}{\bf Ooze decay.} We assume that ooze decays at a constant rate. Originally, this was to represent that ooze is no longer viable for infection after one or two years \citep{slack2017news}, however, it can also be interpreted as a rate that describes how fast the ooze becomes inaccessible for insect dispersal, perhaps by human removal or through environmental factors. That is, it can be interpreted also as a parameter to incorporate disease management practices. In our simulation experiments, we treat it as originally intended.

\item\label{A7}{\bf Infection.} The rate at which flowers transition from susceptible to infectious is dependent on the local floral pathogen population. We further assume that in order for infection to occur, there needs to be some threshold amount of the pathogen exceeded.

\item\label{A8} {\bf Death.} We assume that the rate at which flowers die is a function of how  strongly infected the population is, such that a greater number of infected flowers will weaken the local host population as a whole. This is to reflect the fact that these flowers are connected via a cluster, shoot or branch, and since the bacteria systematically invades the tree, the health of each host is in some sense dependent on the health of all the other hosts at the same location. We further assume that when only a small number of hosts are infected at a location, the average death rate of the hosts is close to zero, but when the infected population exceeds some threshold, the average death rate approaches some maximum value.

\item\label{A9}{\bf Transport.} We assume that the pathogen within the floral cup is transported by pollinators that move randomly throughout the orchard. We make the same assumption for the transport of ooze, but that this process is facilitated by non-bee insects such as flies. We do not account for the dissemination of the pathogen due to wind, rain or human. 
In contrast to the pathogen, the host population (flowers) is immobile and assumed to remain stationary at all times.  
 
\item\label{A10}{\bf Ooze conversion.} The pathogen can be transferred from living within the secreted ooze to living on the surface of flower. We assume that this liberation process is dependent on the rate at which ooze-carrying vectors visit flowers, the amount of ooze at that location, and the density of healthy flowers at that position. The underlying assumption here is that ooze-carrying vectors prefer to feed on the secreted ooze of infected flowers over the nectar of healthy flowers, and are more likely to deposit greater amounts of ooze to flowers (by pollinating for longer at that location) when there are a greater number of healthy flowers available.

\item\label{A11}{\bf Boundary conditions.} We assume that the pathogen does not leave the orchard by the transport of insects and the pathogen cannot be brought into the orchard by insects. In reality, orchards are surrounded by a variety of ecosystems, which makes the choice of proper boundary conditions highly context dependent, so to some extent this is an arbitrary choice that we are forced to make here without introducing more specific situational detail.
\end{enumerate}

\subsection{Governing equations}

\subsubsection{Model formulation}

Based on the assumptions \ref{A1}-\ref{A11} we propose the following dynamical model for the spread of fire blight in orchards during bloom. In accordance with \ref{A1} the model is spatially explicit. For simplicity of notation and mathematical treatment we consider first the one dimensional setting, but point out that the model itself will be straightforward to extend to the 2D setting. The model is then formulated in terms of the independent variables location $x$ (meters) and time $t$ (days). The model domain is then 
\begin{equation*}
    \Omega := (x,t) \in [0,L] \times [0,+\infty) \ .
\end{equation*}
The dependent variables for our model are density of the bacterial populations that are dispersed in the domain, and the host population that is stationary. In accordance with assumption \ref{A2} above, both will be subdivided. The former consists of the bacteria existing within the floral cup,  $B(x,t)$ (CFUs per meter), and  the concentration of the pathogen in the form of ooze by $O(x,t)$ (CFUs per meter). 
The hosts are the blossom flowers, which we subdivide into susceptible to the disease, $S(x,t)$, infected by the disease, $I(x,t)$, or removed, $R(x,t)$, i.e dead flowers. These all are represented by their density (number of flowers per meter).
The spatio-temporal model then reads

\begin{eqnarray}
          \partial_t B&=&  \underbrace{D_1 \partial_x^2 B}_{\text{A9}} +\underbrace{rB\left(1-\frac{B}{K(S+I)+\epsilon}\right)}_{\text{A4}} +\underbrace{\mu OS}_{\text{A10}}  \label{B:pde}\\ 
          \partial_t O &=& \underbrace{D_2 \partial_x^2 O}_{\text{A9}} + \underbrace{\alpha I}_{\text{A5}} - \underbrace{\mu OS}_{\text{A10}} - \underbrace{\gamma O}_{\text{A6}} \label{O:pde} \\ 
    \partial_t S &=& \underbrace{-  f(B) S}_{\text{A7}} \label{S:ode}\\ 
    \partial_t I &=& \underbrace{f(B)S}_{\text{A7}}- \underbrace{g(I) I}_{\text{A8}} \label{I:ode}\\ 
    \partial_t R&=& \underbrace{g(I)I}_{\text{A8}} \ . \label{R:ode}
\end{eqnarray}

This PDE-ODE coupled system of equations is to be augmented with initial conditions for all dependent variables, and with boundary conditions for~$B$ and~$O$. 
At time $t=0$ we have
$$S(x,0)=S_0(x), \quad I(x,0)=I_0(x),\quad R(x,0)=R_0(x)\ , $$
all of which are non-negative.
As per assumption \ref{A3},  the host population is initially evenly-distributed, i.e
\begin{equation*}
    S_0(x)+I_0(x)+R_0(x)=N \ \ \text{(Flowers per meter)}\ \ \text{for all }x \in [0,L]\ ,
\end{equation*}
where $N$ is the flower density.
We remark that this is not a strict requirement for the analyis, nor for the simulations, for which it suffices that each initial data is bounded. 
The initial conditions for the pathogen populations, $B(x,0)=B_0(x)$ and $O(x,0)=O_0(x)$, are assumed to be non-negative bounded and uniformly continuous.

Per \ref{A11} we assume that insects do not carry the pathogen into or out of the orchard. This was motivated by the idea that the amount of bacteria that an insect can transport across the boundaries of our domain is negligible in comparison to the total amount of bacteria in an orchard infected with fire blight. 
Many different assumptions regarding the behaviour of the pathogen at the boundary could have been made, each of which would result in different boundary conditions for our differential equations. Dirichlet conditions, for example, could be used to describe a source of the pathogen existing at the end of a row, or perhaps a hostile environment in which the pathogen cannot survive. Alternatively, choosing Robin type conditions could allow us to model both a pathogen source existing outside of the orchard as well as the in-flow and out-flow of the bacteria due to natural pollinators. Of course, we know that for a wholly susceptible orchard to become infected there must be an in-flow of the pathogen from somewhere, but since the nature of primary infection can vary drastically, this becomes a difficult process to model. In this work, the assumption we made is that homogeneous Neumann conditions hold, i.e.
\begin{equation}\label{BC:eq}
  \partial_x B(0,t) = \partial_x B(L,t) =0, \quad
     \partial_x O(0,t) = \partial_x O(L,t) =0.
\end{equation}

By doing this, we are essentially modelling an orchard that exists in a place of isolation, such that the in-flow and out-flow of the pathogen is minimal enough to be neglected. This is clearly not an ideal nor accurate representation of an agricultural system, but it does allow for less guesswork in the modelling process and a simpler mathematical analysis.

The first terms on the right hand sides of (\ref{B:pde}) and (\ref{O:pde}) model the dispersion of the pathogens by insects, per assumption \ref{A9}. Going forward we will use that 
\begin{equation}\label{D1gtD2:eq}
D_1 \geq D_2 \ ,  
\end{equation}
because $B$ is transported mainly by managed pollinators, primarily honeybees. In a commercial setting, we will assume that those are present in sufficient numbers that they visit flowers reliably and fast, and are more efficient at visiting flowers across the orchard than other insects are.

In \ref{A4}, the assumption is that the pathogen population size would grow exponentially
provided that there was access to unlimited resources for survival and reproduction, but this population size is ultimately limited by a carrying capacity dependent on the health of the flower. These dynamics give rise to a sort of logistic growth curve for the pathogen, which would be consistent with the experiments in \citet{wilson1989erwinia, wilson1990erwinia}, but inconsistent with the idea that bacteria do not grow as well on older flowers as suggested in \citet{slack2022orchard}. Mathematically, we translate this assumption into the logistic term in (\ref{B:pde}). Here, $r$ (1/day) is the maximum bacterial growth rate. The small parameter $\epsilon$ (CFU/meters) is mathematically a regularisation parameter; biologically it is the density of surface associated pathogens that can be sustained on the trees in the absence of flowers. This carrying capacity increases in the presence of flowers. The term $K(S+I)$ describes how many bacteria can be sustained in the presence of viable hosts; parameter $K$ has the units CFU/flowers.

Per assumption \ref{A10} bacteria in the ooze are converted into the flower associated form at the same location, i.e. a transfer between pathogen fractions~$O$ and~$B$. We model this process to be proportional to both, the ooze density and the density of uninfected flowers. This is the last term in (\ref{B:pde}), where it is a source term, and the third term in (\ref{O:pde}), where it is a sink term. The parameter $\mu$ (meters/(flowers*day)) is the rate at which ooze is deposited onto healthy flowers per one increase in susceptible host density. 

Per \ref{A5}, ooze production in (\ref{O:pde}) is related to current infection level, which is represented in our model by density of infected flowers.  It is important to note that not all infected blossoms necessarily ooze, and so we are also assuming that the environmental conditions are highly conducive to the production and secretion of ooze, such that every infected flower does so.  We assume this to be subsumed in the parameters of the ooze production function $h(I)$. We know that there can be no secretion of ooze if there are no flowers actively infected with the pathogen, so we require $h(0)=0$. It would also seem natural to assume that this function is a monotonically increasing function of $I$. In the absence of more specific information and with the goal of limiting model complexity, we assume that ooze production is proportional to the infected host population, $h(I)=\alpha I$. The parameter $\alpha$ (CFUs per flower per day) can be interpreted as the average amount of ooze produced by an infected flower each day. We assume that $\alpha$ is a biological property of the tree and of the pathogen and that this parameter does not change over time.  Per \ref{A6} we assumed that ooze looses viability at a constant rate $\gamma$ (per day), which introduces the linear sink term in (\ref{O:pde}).

Infection in our model follows a two-step process, as is common for many infectious disease models. Per \ref{A7} the infection rate depends on the density of host associated pathogen, which we denote by $f(B)$ in (\ref{S:ode}), (\ref{I:ode}).  This infection incurs if the pathogen density exceeds a certain threshold. Assuming that the expression of virulence factors, as it is often the case for bacterial diseases, is controlled by a quorum sensing mechanism. To allow for cell and environment specific variability in the expression of virulence factors, we propose to model the dependence of the infection on bacterial density by a smooth Hill (rather than Heaviside) function,
\begin{equation}\label{ffun:eq}
    f(B)= M_1 \frac{B^{n_1}}{A_1^{n_1}+B^{n_1}}\ , 
\end{equation}
where parameter $M_1$ (per day) is the maximum infection rate, and $A_1$ (CFUs per meter) is the average critical population density, that controls expression of virulence factors. The Hill exponent $n_1>1$ controls how swift the switch from an almost negligible to a proportional infection rate is.
 Similarly, per \ref{A8}, we assume that the dying off of flowers due to the disease depends on the infection level of the host, i.e. on~$I$. Thus we have this expressed by a function~$g(I)$ in~(\ref{I:ode}) and~(\ref{R:ode}).  Assuming that at low infection levels this is negligible, but becomes proportional to the density of infected hosts if the number of infected hosts is large, we again introduce a smooth threshold function for this process, namely
\begin{equation}\label{gfun:eq}
    g(I)= M_2 \frac{I^{n_2}}{A_2^{n_2}+I^{n_2}}\ . 
\end{equation}
Here the constant parameter $M_2$ (per day) is the maximum host death rate, and $A_2$ (flowers per meter) is the threshold parameter that describes the switch from an almost negligible to a proportional death rate. The Hill exponent $n_2>1$ controls how swift this switch is. Going forward, in our simulations and analysis, we will use 
\begin{equation*}
n_1=n_2=2
\end{equation*}
for convenience. We do not think that fixing these parameters has a major influence on model behavior.

As pointed out above, the model we propose is for the period of bloom, which typically lasts only a few weeks. Therefore the time scale that we consider is of the order of days.  This suggests that the usual analytical methods that are commonly employed to study long term behaviour of such dynamical models as $t\rightarrow \infty$ might be limited in what they can predict. A full list of the model parameters and their units are presented in Table \ref{param_tab}.

\begin{table}
\centering
\caption{Model parameters, their description and units.}
\begin{tabular}{||p{0.12\linewidth} | p{0.5\linewidth} |p{0.28\linewidth} ||}
 \hline
 Parameter & Description & Units \\ [0.5ex] 
 \hline\hline
 $N$ & Average number of flowers at each location. & Flowers per meter.\\
 \hline
 $D_1$ & Diffusion coefficient for pathogen dispersal by bees. & $\text{Meters}^{2}$ per day. \\ 
 \hline
$D_2$ & Diffusion coefficient for ooze dispersal by non-bee vectors. &$\text{Meters}^{2}$ per day. \\
 \hline
 $K$  & \textit{E.amylovora} carrying capacity of a living flower.  &CFUs per flower.\\
 \hline
  $\epsilon$  & Carrying capacity of a completely dead flower population.  & CFUs per meter. \\
 \hline
 $r$ &  Growth rate of floral pathogen population.  & Per day.\\
\hline 
 $\mu$ & Pathogen liberation rate.  & Meters per flower per day.\\
\hline 
 $\gamma$ & Ooze decay rate.  & Per day. \\
\hline 
$\alpha$ & Ooze production rate.  & CFUs per flower per day.\\
\hline
$M_1$ & Maximum infection rate.  & Per day. \\
\hline 
$M_2$ & Maximum death rate.  & Per day.\\
\hline
$A_1$ & Threshold parameter for invasion of \textit{E.amylovora}.  &CFUs per meter.\\
\hline
$A_2$ & Threshold parameter for the death of flowers.  & Flowers per meter. \\
\hline
$n_1$ & Hill function exponent for invasion function.  &--- \\
\hline
$n_2$ & Hill function exponent for death-rate function.   &--- \\ [1ex] 
 \hline
\end{tabular}
\label{param_tab}
\end{table}

\subsubsection{Well-posedness of the initial boundary value problem}\label{exist:sec}

\begin{lemma} \label{lemma:non-neg}
Solutions to system (\ref{B:pde})-(\ref{BC:eq}) with (\ref{ffun:eq}), (\ref{gfun:eq})  remain non-negative for all time $t\in [0,+\infty)$ and are bounded from above by positive constants.
\end{lemma}

\begin{proof} First we show that the ODE system remains non-negative and bounded above by positive constants. This follows with standard arguments for ODEs. More specifically, we find that if $N:=\max\limits_{0\leq x \leq L}  S_0(x)+I_0(x)+R_0(x)$, then
\begin{align*}
0 \leq S_0(x) e^{-M_1t} \leq S(x,t)  &\leq N ,\\
0 \leq I_0(x) e^{-M_2t} \leq  I(x,t)  &\leq N ,\\
0 \leq R_0(x)\leq  R(x,t)  &\leq N .
\end{align*} 

Using the non-negativity and upper bounds of $S,I$, the parabolic comparison theorem
and spatially uniform solutions to the associated ordinary differential inequalities we find in the usual manner  for $0\leq x \leq L, 0\leq t$ that

\begin{equation*}
   \min\limits_{0\leq x \leq L} O_0(x) e^{-(\mu+\gamma)\cdot N \cdot t}  \leq  O(x,t) \ \leq \ \max\left\{ \max\limits_{0\leq x \leq L} O_0(x) , \frac{\alpha N}{\gamma}\right\} \ \ \ \ \text{for all} \  \ (x,t) \in \Omega \ ,
\end{equation*}
and similarly, the bounds on $O(x,t), S(x,t), I(x,t)$
guarantee that
\begin{equation*}
    0\leq B(x,t)\leq  \max \  \left \{ \max\limits_{0\leq x \leq L} B_0(x) ,  \ \frac{K\cdot N+\epsilon}{2} \left( 1 + \sqrt{1+\frac{4\mu \max\left\{ \max\limits_{0\leq x \leq L} O_0(x) , \frac{\alpha N}{\gamma}\right\}}{r(K\cdot N+\epsilon)}} \right) \ \right \} ,
\end{equation*}
which again is derived by comparison against spatially uniform solutions. 
\end{proof}

With this in hand we can state the following existence and uniqueness result for our model on the bounded domain. Let $BC \cap L^{\infty}$ denote the class of functions that are bounded and uniformly continuous on $\mathbb{R}$ and $L^{\infty}$-integrable.

\begin{theorem} \label{theorem: 2.0.2}
For a set of non-negative initial conditions where
\begin{equation*}
    S_0(x)+I_0(x)+R_0(x)=N \ \ \text{for all } x \in [0,L],
\end{equation*}
and
\begin{equation*}
    B_0(x),O_0(x),S_0(x),I_0(x),R_0(x) \in   BC \cap L^{\infty} \ ,
\end{equation*}
system (\ref{B:pde})-(\ref{BC:eq}) with (\ref{ffun:eq}), (\ref{gfun:eq})  admits a unique solution on any time interval $t \in [0,T]$ such that 
\begin{align*}
    B(x,t),\ O(x,t),\ S(x,t),\ I(x,t),\ R(x,t) \in C([0,T],  BC \cap L^{\infty}) \ .
\end{align*}

\end{theorem}

\begin{proof} (outline) The ideas of this proof come directly from Chapter 14 of \citep{smoller2012shock}. It follows from our Lemma \ref{lemma:non-neg} that the solutions of system arising from the constraints on our initial conditions remain non-negative and bounded above by the constants on the right hand side of each inequality for all time. Using Definition 14.1 and Theorems 14.2-14.4 in Smoller's book \citep{smoller2012shock}, one can define a contraction mapping in a Banach space, from which one can apply Banach's fixed point theorem to show that a unique fixed point to the mapping exists. This fixed point gives a unique solution for the system up to some time $t_0$. The time $t_0$ depends only on the non-linearity of the system and the a-priori upper bound on the initial conditions, and since the solutions remain bounded for all time by these constants, we can iterate this mapping by taking the values of the solution at $t_0$ as initial conditions to extend this solution over any time interval $[0,T]$, giving a unique solution 
$B, O, S, I , R \in C([0,T),  BC \cap L^{\infty}).$
\end{proof}

\subsubsection{Disease invasion in the fire blight model: Exploratory simulations}\label{sim:sec}

To explore the behaviour of solutions of our model we conducted first numerical simulations, using the R language \citep{Rcite}, more specifically the function ``ode.1d"  of the package {\tt deSolve} \citep{soetaert2010solving}, which applies finite differencing to discretize in space. We used the implicit Adams scheme for time integration.  

In Figure~\ref{illustTW}, we present one particular simulation that is representative of what we typically saw for a variety of parameter sets. More specifically, we present snapshots in time after 35, 40, and 45 days, i.e. within a time period of 5 to 7 weeks.
The parameters for this figure were chosen more or less arbitrarily for the point of illustration, and are listed in Table \ref{TWtable} below. For initial conditions, we began with an ooze concentration of $10^4$ CFUs per meter at $x=0$ and nowhere else. All other model components were set to be zero everywhere, except for the susceptible population, which was set to a constant value $N$ at each point. 

\begin{table}
\centering
\caption{Model parameters used in the exploratory simulations presented in Figures~\ref{illustTW}.}

\begin{tabular}{||p{0.5\linewidth} | p{0.08\linewidth} |p{0.08\linewidth} | p{0.2\linewidth} ||}
 \hline
 Parameter & Symbol& Value & Units \\ [0.5ex] 
 \hline\hline
  Average number of flowers at each location. & $N$ & $5$ &Flowers per meter.\\
  \hline
 Diffusion coefficient for pathogen dispersal by bees. & $D_1$  & $20$ & $\text{Meters}^2$ per day. \\ 
 \hline
 Diffusion coefficient for ooze dispersal by non-bee vectors. & $D_2$ & $10$ &  $\text{Meters}^2$ per day. \\ 
 \hline
  \textit{E.amylovora} carrying capacity of a living flower.& $K$  & $10^6$ & CFUs per flower. \\ 
 \hline
  Carrying capacity of a completely dead flower population. & $\epsilon$   & $10^5$ & CFUs per meter. \\ 
 \hline
Growth rate of floral pathogen population. & $r$  & $0.5$ & Per day. \\ 
 \hline
Ooze conversion rate.  & $\mu$     & $0.5$ & Meter per flower per day. \\
 \hline
Ooze decay rate.  & $\gamma$   & $1$ & Per day.\\
 \hline
   Ooze production rate.  & $\alpha$   & $10^8$ & CFUs per flower per day.\\
 \hline
    Maximum infection rate. & $M_1$   & $1$ & Per day.\\
 \hline
     Maximum death rate. & $M_2$  & $1$ & Per day. \\
 \hline
   Threshold parameter for invasion of \textit{E.amylovora}.  & $A_1$  & $ 10^6$ & CFUs per meter. \\
 \hline
   Threshold parameter for the death of flowers. & $A_2$    & $1$ & Flowers per meter. \\
 \hline
    Hill function exponent for invasion function. & $n_1$    & $2$ & ---\\
 \hline
     Hill function exponent for death-rate function.  & $n_2$     & $2$ & ---\\
 \hline
\end{tabular}
\label{TWtable}
\end{table}

\begin{figure}
\includegraphics[width=\textwidth]{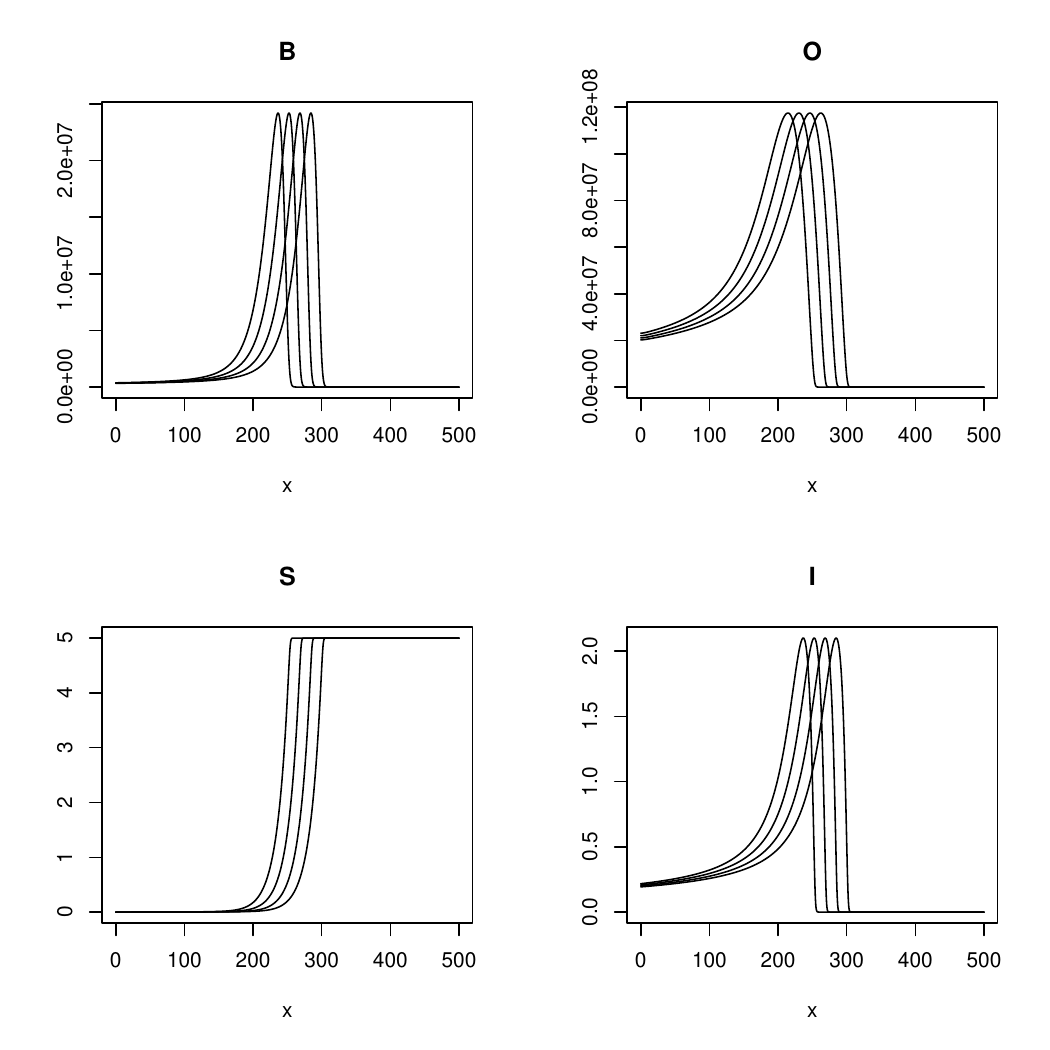}
\caption{An illustrative simulation to show the typical dynamics of the model. Shown are $B,O,S,I$  at three equally spaced time points ; the parameters used for this simulation are summarized in Table \ref{TWtable}. Not shown are the data for $R$, since $R=N-S-I$ decouples from the system.}
\label{illustTW}
\end{figure}

The visualisation of the simulation in Figure~\ref{illustTW} shows that both pathogen populations $B$ and $O$ invade the initially disease free domain. After some initial transient period, for given $t$ there is a sharp peak for both $B$ and $O$ at about the same location, and behind these peaks both trail off. The  floral pathogen population to a value of $\epsilon$, the ooze bound pathogen population eventually to a value of $0$, but much slower than $B$. The infected host population shows similar behaviour, with a steep incline at the wave front and a slower decay in its wake.  The susceptible host population $S(x,t)$ decreases in time monotonically from $N$ in front of the wave  to 0 in its wake with depletion where the disease originates. The behaviour of $R(x,t)$ (data not shown) mirrors the behaviour of $S(x,t)$: it is monotonically decreasing in space with no dead hosts ahead of the pathogen wave and an entirely dead host population behind it. 

Upon closer inspection, it appears that each component of this solution takes a shape that does not change, but moves with a constant speed. This suggests that the invasion of the domain by the disease might be described by a travelling wave solution of the model. Noteworthy here is that two of the components are monotonic (susceptible and removed host populations~$S$ and~$R$), two are pulse waves that drop to 0 after the peak (ooze population $O$ and infected hosts $I$), and one is a pulse wave that levels off at  constant value in the wake of the peak (floral pathogen population $B$).  The solution profile appears to already have been established at the first time instance shown. This suggests that its formation, i.e. the invasion of an initially disease free orchard, happens quickly, on a time scale of only a few weeks which coincides with the intended time frame of bloom for which the model was developed.  Motivated by these simulations, we set out to investigate further whether our model permits travelling wave solutions, and under which conditions on parameters.

\section{Travelling waves analysis}

\subsection{Critical wave speed for $n_1=2$}

A travelling wave $(B,O,S,I)$ of (\ref{B:pde})-(\ref{I:ode}) with (\ref{ffun:eq})-(\ref{gfun:eq}) is an entire in time solution depending only on a moving variable $z:=x-ct$, where the wave speed $c$ is to be determined. These solutions, as functions of $z$, satisfy
 \begin{align}
-cB'(z) - D_1 B''(z) &= rB(z)\left(1-\frac{B(z)}{K(S(z)+I(z))+\epsilon}\right) + \mu \cdot O(z)\cdot S(z),\label{B:tweq} \\
-cO'(z) - D_2O''(z) &= \alpha \cdot I(z) - \mu \cdot O(z)\cdot S(z) - \gamma O(z),  \label{O:tweq}\\
-cS'(z) &= -f(B(z))S(z), \label{S:tweq}\\
-cI'(z) &= f(B(z))S(z) - g(I(z))I(z) \ . \label{I:tweq}  
\end{align}
They should be positive and connect the homogeneous steady states
\begin{equation}\label{asymp+}
(B,0,S,I)  (+\infty) = (0,0,N,0)\ 
\end{equation}
and
\begin{equation}\label{asymp-}
(B,0,S,I)  (-\infty) = (\varepsilon,0,0,0)\ ,
\end{equation}
corresponding respectively to the orchard's state before and after the disease outbreak. We omit in our analysis the equation for $R$, as it decouples from the system.

We conjecture that there exists some $c^* >0$ such that a travelling wave exists if and only if $c \geq c^*$. The value of $c^\ast$ should be expected to depend on model parameters.  Below we only address existence for non critical values $c > c^*$. As usual, the critical wave speed $c^*$ is determined by linearization around the invaded steady state $(0,0,N,0)$. 
The linearized system writes as
\begin{equation}\label{linsys:eq}
\left\{
\begin{array}{l}
 \partial_t B = D_1 \partial_x^2  B + r B  + \mu N O \ , \vspace{3pt}\\
\partial_t O = D_2 \partial_x^2 O + \alpha I - \mu N O - \gamma O\ , \vspace{3pt}\\
\partial_t S = - f' (0) N B\ , \vspace{3pt}\\
 \partial_t  I=  f'(0) N B \ .
\end{array}
\right.
\end{equation}

Due to the model assumption \ref{A7} and the subsequent discussion of (\ref{ffun:eq})  we restrict ourselves to the case $n_1 >1$, where we have $f'(0) = 0$. Going forward, for simplicity we will pick $n_1 = 2$ in our calculations, as we did previously in the sections \ref{exist:sec} and \ref{sim:sec}, however, we conjecture that our argument of this section will also apply to $n_1 \geq 2$.  
Thus we are left with
\begin{equation*}
\left\{
\begin{array}{l}
 \partial_t B = D_1 \partial_x^2  B + r B  + \mu N O\ ,  \vspace{3pt}\\
\partial_t O = D_2 \partial_x^2 O - \mu N O - \gamma O\ . \vspace{3pt}\\
\end{array}
\right.
\end{equation*}
Next, we make an exponential {\it ansatz}
$$(B,O)=(B_\lambda, O_\lambda)\  e^{-\lambda (x-ct)}\ , $$
Upon substituting this into (\ref{linsys:eq}) we obtain the triangular system:
\begin{equation*}
\left\{
\begin{array}{l}
 c \lambda B_\lambda = D_1 \lambda^2 B_\lambda + r B_\lambda  + \mu N O_\lambda\ , \vspace{3pt}\\
 c \lambda O_\lambda = D_2 \lambda^2 O_\lambda  - \mu N O_\lambda - \gamma O_\lambda \  .
\end{array}
\right.
\end{equation*}
This leads us to introduce the matrix
$$M_\lambda = \left( 
\begin{array}{cc}
D_1 \lambda^2 + r & \mu N \\
0 & D_2 \lambda^2 - \mu N - \gamma 
\end{array}
\right)\ 
,
$$
and to compute
$$\Lambda (\lambda) := \max \{ D_1 \lambda^2 +r , D_2 \lambda^2 - \mu N - \gamma \}\ ,$$
the largest (positive) eigenvalue of $M_\lambda$, which by the Perron-Frobenius theorem (or an elementary computation) is associated with a positive eigenvector.

We then define
$$c^* := \min\limits_\lambda\left\{ \frac{\Lambda (\lambda)}{\lambda}\right\} >0\ ,$$
which according to the above is the minimal speed for such a (positive) exponential ansatz to solve the linearized problem. By analogy with the standard reaction-diffusion equation with logistic growth, or Fisher-KPP~\citep{fisher1937wave,kolmogorov1937study}, we conjecture that this is the minimal traveling wave speed.
Due to model assumption~\ref{A9} and $D_1 \geq D_2$ as in (\ref{D1gtD2:eq}), we have
$\Lambda(\lambda)= D_1 \lambda^2 +r$ and after some straightforward calculations, we find that
$$c^* = 2 \sqrt{D_1 r}\ .$$
This suggests that the travelling wave speed and exponential decay are fully determined by the parameters in the governing equation for $B$. 

This value $c^\ast$ is already known to be the minimum travelling wave speed for the classical Fisher equation~\citep{fisher1937wave,kolmogorov1937study}, that is
$$\partial_t u = D_1 \partial_x^2 u + ru \left( 1- \frac{u}{\epsilon}\right)\ .$$
One can check that such a solution~$u$ is also a sub-solution for the $B$-equation~\eqref{B:pde} in our fire blight model. Due to well-known spreading properties to the minimal travelling wave for the Fisher equation~\citep{aronson1975}, it follows that our fire blight model cannot admit a travelling wave with speed strictly less than $c^*$.

Therefore, going forward we will only be concerned with the existence of travelling waves for any speed (strictly) larger than $c^* = 2 \sqrt{D_1 r}$.

\subsection{Existence of travelling waves}
Our main result below will be the proof of the existence of travelling wave solutions to (\ref{B:pde})-(\ref{I:ode}) with nonlinearities (\ref{ffun:eq})-(\ref{gfun:eq}). For this we use a fixed point argument, similar to those used e.g. in \citep{ma2001traveling,bhkr2005,huang2006travelling,shang2016traveling,shu2019traveling,zhou2019critical}. This method involves the construction of suitable upper and lower solutions, the construction of a continuous and compact mapping on a set defined by these upper and lower solutions, and the application of Schauder's fixed point theorem. Here we adapt this argument to the PDE-ODE coupled system (\ref{B:pde})-(\ref{I:ode}), which is an everywhere degenerate semilinear system of reaction-diffusion equations. We begin with establishing in Section \ref{UpLoBounds:sec} upper and lower solutions that are specific to the model at hand, and then define an integral map in Section \ref{integralmap:sec} on the set bounded by the upper and lower solutions, the fixed points of which satisfy the travelling wave equations. The main result will then be in Section \ref{twexistence:sec} an existence theorem for travelling wave solutions.

Throughout this analysis, we will consider
$$c > c^* = 2 \sqrt{D_1 r}\ .$$

\subsubsection{Upper and lower Solutions}\label{UpLoBounds:sec}

We need to find two quadruplets
$$0 \leq (\underline{B},\underline{O},\underline{S},\underline{I})   \leq (\overline{B},\overline{O},\overline{S},\overline{I}) \ ,$$
where the inequalities are to be understood componentwise, satisfying (possibly in a generalized sense)
\begin{eqnarray}
\mathcal{L}_B & := &  D_1 \underline{B}'' + c \underline{B}' +r \underline{B} \left( 1 -\frac{\underline{B}}{K(\underline{S}+\underline{I}) + \varepsilon} \right) + \mu \underline{O}\underline{S} \geq 0  \label{LB:ineq}\\
\mathcal{U}_B & := &  D_1 \overline{B}'' + c \overline{B}' +r \overline{B} \left( 1 -\frac{\overline{B}}{K(\overline{S}+\overline{I}) + \varepsilon} \right) + \mu \overline{O}\overline{S} \leq 0 \label{UB:ineq}\\
\mathcal{L}_O & := &  D_2 \underline{O}'' + c \underline{O}' + \alpha \underline{I} - \mu \underline{O} \overline{S} - \gamma \underline{O} \geq 0\label{LO:ineq}\\
\mathcal{U}_O & := &  D_2 \overline{O}'' + c \overline{O}' + \alpha \overline{I} - \mu \overline{O} \underline{S} - \gamma \overline{O} \leq 0 \label{UO:ineq}\\
\mathcal{L}_S & := &  c \underline{S}' - f ( \overline{B}) \underline{S} \geq 0 \label{LS:ineq}\\
\mathcal{U}_S & := & c \overline{S}'  - f (\underline{B}) \overline{S} \leq 0 \label{US:ineq}\\
\mathcal{L}_I & := &  c \underline{I} ' + f ( \underline{B}) \underline{S} - g (\underline{I}) \underline{I} \geq 0 \label{LI:ineq}\\
\mathcal{U}_I & := &  c \overline{I} ' + f (\overline{B}) \overline{S} - g(\overline{I}) \overline{I} \leq 0 \ . \label{UI:ineq}
\end{eqnarray}
By a generalized sense, we mean that some of our lower and upper solutions below will satisfy those inequalities except at a finite number of points where they will lack the necessary regularity. However, provided that their left and right derivatives at those points are correctly ordered, the comparison principle underlying the fixed point approach remains available (see the later subsections and Lemma~2.1 in~\citet{shu2019traveling}).

Recall that $c > c^* = 2 \sqrt{D_1 r}$. Then there exists~$\lambda$ the smallest positive solution of
$$\Lambda (\lambda) =D_1 \lambda^2 + r = c \lambda\ .$$
We first define the following candidate functions for the upper solutions.
$$\overline{S} = N \ , $$
$$\overline{B} =  \min \left\{ B_0,  e^{-\lambda z} \right\} \  ,$$
$$\overline{O} =  \min \left\{ \frac{\alpha I_0}{\gamma}, \eta e^{-2 \lambda z} \right\}\ ,$$
and $$\overline{I} =  \min \left\{ I_0 , \int_z^{\infty} \frac{ N f (\overline{B})}{c} \right\}\ ,$$
where $I_0$ is such that $g(I_0) I_0 = M_1 N$, and $B_0$ is the unique positive solution to
$$r B_0 \left(1 - \frac{B_0}{K(N+I_0) + \epsilon } \right) + \mu \frac{\alpha I_0}{\gamma} N = 0 \ .$$ The constant $\eta$ will be chosen later in order to ensure that the required inequality~(\ref{UO:ineq}) will indeed be satisfied.

For later use, we also compute that
\begin{equation}\label{overlI:eq}
\overline{I}  \leq   \int_z^{\infty} \frac{ M_1 N \overline{B}^2 }{A_1^2 c}  \leq   \frac{M_1 N}{2A_1^2 \lambda c} e^{-2 \lambda z}\ ,
\end{equation}
owing to the properties of the function~$f$ in~\eqref{ffun:eq}.

As candidate functions for lower solutions, we define
$$\underline{S} = N  e^{ - \int_{z}^{\infty} \frac{f(\overline{B})}{c}}\ ,$$
$$\underline{B} =  \max \{ 0,  e^{-\lambda z} - \nu e^{-(\lambda + \delta)z} \}\ ,$$
$$\underline{O}=  0\ ,$$
and
$$\underline{I} = 0\ .$$
Here the new parameters $\delta$ and $\nu$ in the expression for $\underline{B}$ are chosen to satisfy 
\begin{equation}\label{delta:eqi}
\delta \in (0, \lambda),   \quad  D_1 (\lambda+\delta)^2 - c (\lambda+\delta) +r < 0\ , 
\end{equation}
and
\begin{equation}\label{nu:eqi}
\nu >  \max  \left\{ 1 , \frac{r}{\varepsilon} \times \frac{1}{ c( \lambda+\delta) - r - D_1 (\lambda+\delta)^2} \right\} >0\ .
\end{equation}
Notice that $\underline{S}\leq \overline{S}$, $\underline{O} \leq \overline{O}$, $\underline{I} \leq \overline{I}$, and also $\underline{B}\leq \overline{B}$ up to increasing $\nu$ and without loss of generality.

Next we verify that these functions indeed satisfy the inequalities (\ref{LB:ineq})-(\ref{UI:ineq}) for a suitable choice of $\eta$ and under some additional condition on $\mu$ and $\alpha$ which will be specified below.

\paragraph{Computation of $\mathcal{U}_S$.} This is trivial due to $\overline{S} =N$ and then
$$\mathcal{U}_S = - N f (\underline{B} ) \leq 0\ .$$

\paragraph{Computation of $\mathcal{U}_I$.}
When $\overline{I} < I_0$, then we have $c \overline{I}' + N f( \overline{B}) = 0$ by definition, hence
$$ \mathcal{U}_I =  c \overline{I} ' + f(\overline{B}) \overline{S} - g( \overline{I})\overline{I} \leq 0\ .$$
On the other hand, $I_0$ satisfies
$$f(\overline{B}) \overline{S} - g (I_0)I_0 = N ( f(\overline{B}) - M_1) \leq 0\ .$$

\paragraph{Computation of $\mathcal{U}_O$.} When $\overline{O} = \frac{\alpha I_0}{\gamma}$, then
$$\mathcal{U}_0 = \alpha \overline{I} - \mu \frac{\alpha I_0}{\gamma} \underline{S} - \gamma \frac{\alpha I_0}{\gamma} \leq \alpha (\overline{I} - I_0 ) \leq 0\ .$$
When $\overline{O} < \frac{\alpha I_0}{\gamma}$, we compute that
\begin{eqnarray}
\mathcal{U}_O & \leq & \eta e^{-2 \lambda z} \times \left[ 4 D_2 \lambda^2 - 2 c \lambda - \gamma \right] + \alpha \overline{I}  \nonumber\\
& \leq & \eta e^{-2 \lambda z} \times \left[ 4 D_2 \lambda^2 - 2 c \lambda - \gamma \right] +  \frac{\alpha M_1 N}{2A_1^2 \lambda c} e^{-2 \lambda z}\  . \label{UOinequ:eq}
\end{eqnarray}
We recall that we yet have to choose the constant $\eta$. In order to achieve the desired inequality for $\mathcal{U}_O$ we propose
\begin{equation}\label{eta:eq}
\eta = \frac{\alpha M_1 N }{2A_1^2 \lambda c \times (\gamma + 2 c \lambda  - 4 D_2 \lambda^2)}\ .
\end{equation}
It is easy to verify that then $\eta$ is indeed positive: $\lambda$ is the smallest positive solution of $D_1 \lambda^2 +r = c \lambda$, hence $\lambda = \frac{c - \sqrt{c^2 - 4 D_1 r}}{2D_1} \in \left(0, \sqrt{\frac{r}{D_1}} \right)$, and we have  
\begin{eqnarray*}
4 D_2 \lambda^2 -2  c \lambda - \gamma  & = & (4D_2 - 2 D_1) \lambda^2 - 2r - \gamma\ .
\end{eqnarray*}
If $4D_2 - 2 D_1 \leq 0$, then the right-hand term is negative regardless of $\lambda$. If $4D_2 - 2 D_1 > 0$, then we have
\begin{equation}\label{etapos:eq}
4 D_2 \lambda^2 - 2 c \lambda - \gamma \leq (4D_2 - 2 D_1 ) \frac{r}{D_1} - 2r - \gamma < 0\ ,
\end{equation}
where the last inequality holds due to our model assumption (\ref{D1gtD2:eq}) that $D_2 \leq D_1$.

Substituting (\ref{eta:eq}) into (\ref{UOinequ:eq}) we obtain indeed $\mathcal{U}_O \leq 0$, as required.

\paragraph{Computation of $\mathcal{U}_B$.} The pathogen liberation rate $\mu$ must be suitably bounded in order for us to be able to obtain $\mathcal{U}_B \leq 0$. The key point here is to derive conditions such that the dynamics of the system is determined by the dynamics of the Fisher type equation for $B$, for example in the form 
\begin{eqnarray}
\mu \overline{O} \overline{S} & \leq & \mu  \eta N e^{-2 \lambda z} \nonumber \\
& \leq & \mu  \eta N \overline{B}^2 \nonumber\\
& \leq & \frac{r \overline{B}^2}{K(N +I_0) +\varepsilon}\ , \label{muO:ineq} 
\end{eqnarray}
where the first two inequalities hold by definition of~$\overline{O}$ and~$\overline{B}$, and the third one holds provided that
$$\mu \alpha < \frac{r}{K(N +I_0) +\varepsilon} \times \frac{2A_1^2 \lambda c (\gamma + 2 c\lambda - 4 D_2 \lambda^2)}{M_1 N }\ ,$$
due to our choice of~$\eta$ in~(\ref{eta:eq}).

That such a positive $\mu$ exists follows again from (\ref{etapos:eq}).
In order for this condition to not depend on $c$ and $\lambda$ but on the original model parameters only, we may instead require
\begin{equation}\label{main_mu_assumption}
\mu \alpha < \frac{4 A_1^2 r^2}{K(N+I_0)+\varepsilon} \times \frac{ 2r\cdot \min \{ 1 , 2 (1-\frac{D_2}{D_1} )\}+ \gamma }{M_1 N}\ ,
\end{equation}
where we recall that $I_0$ is the solution of $g(I_0) I_0 = M_1N$. We point out here that this condition on the pathogen liberation rate allows an ecological interpretation which we will discuss below. Although this constraint on $\mu\alpha$ likely might not be optimal, it provides a sufficient condition for our analysis going forward. 

Under this condition on $\mu$ we obtain~(\ref{muO:ineq}) and then
\begin{eqnarray*}
\mathcal{U}_B & \leq &  D_1 \overline{B}'' + c \overline{B}' +r \overline{B}   \\
& = &  e^{-\lambda z} (D_1 \lambda^2 - c \lambda + r) \\
& = & 0\ .
\end{eqnarray*}

\paragraph{Computation of $\mathcal{L}_S$.} Inequality $\mathcal{L}_S\geq 0$ is satisfied by definition.

\paragraph{Computation of $\mathcal{L}_B$.} The null function trivially satisfies $\mathcal{L}_B \geq 0$, so we only consider the case when $\underline{B} >0$. Due to $\nu > 1$ by~\eqref{nu:eqi}, this implies that $z > 0$. Then we can compute
\begin{eqnarray*}
\mathcal{L}_B & = & e^{-\lambda z} \times (D_1 \lambda^2 - c\lambda +r) - \nu e^{-(\lambda + \delta) z} \times (D_1 (\lambda+\delta)^2 - c(\lambda +\delta) +r) \\
		&&  - r \frac{(e^{-\lambda z} - \nu e^{-(\lambda+\delta)z})^2 }{K (\underline{S} + \underline{I}) +\varepsilon }\\
& \geq & -\nu e^{-(\lambda + \delta) z} \times (D_1 (\lambda+\delta)^2 - c(\lambda+\delta) +r) -   \frac{r e^{-2\lambda z} }{\varepsilon }\\
& \geq &  \nu e^{-(\lambda+\delta) z} \times \left(- D_1 (\lambda+\delta)^2 + c (\lambda+\delta) - r - \frac{r}{\nu \varepsilon} \right) \\
& \geq & 0\ ,
\end{eqnarray*}
where we used our choice of $\delta$ and $\nu$ from \eqref{delta:eqi}-\eqref{nu:eqi} to obtain the last inequality. We only briefly point out that we may have instead used the positivity of $\underline{S}$ on a right half line, to make the same argument work when $\varepsilon = 0$.

\paragraph{Computation of $\mathcal{L}_O$ and $\mathcal{L}_I$.} This is trivial for $\underline{O}=\underline{I} =0$.
\newline

With these upper and lower solutions for the travelling wave system in place, we can proceed to define the required integral map.

\subsubsection{Construction of an integral map whose fixed points are solutions to the travelling wave system}\label{integralmap:sec}

We now define the componentwise differential operator
\begin{equation*}
    \Delta = (\Delta_B, \Delta_O, \Delta_S, \Delta_I) \ ,
\end{equation*}
as
\begin{align*}
    \Delta_B (B)(z) :=& \ -D_1 B''(z) - c B'(z) + \alpha_B B(z), \  \\ 
    \Delta_O (O)(z) :=& \ -D_2 O''(z) - c O'(z) + \alpha_O O(z), \  \\ 
    \Delta_S (S)(z) :=& \  -c S'(z) + \alpha_S S(z), \  \\ 
    \Delta_I (I)(z) :=& \ - c I'(z) + \alpha_I I(z)\  .
\end{align*}
 The constants $\alpha_B,\alpha_O,\alpha_S$ and $\alpha_I$ will need to be taken sufficiently large and will be determined later. The characteristic roots of $\Delta_B$ and $\Delta_O$ are
\begin{align*}
    \lambda_B^\pm :=& \ \frac{-c \pm \sqrt{c^2 + 4D_1 \alpha_B}}{2D_1} \ ,\\ 
    \lambda_O^\pm :=& \ \frac{-c \pm \sqrt{c^2 + 4D_2 \alpha_O}}{2D_2} \ .
\end{align*}
We define the constants
\begin{align*}
\rho_B  :=& D_1(\lambda_B^+ - \lambda_B^-) = \sqrt{c^2+4D_1 \alpha_B}\ , \\
 \rho_O  :=& D_2(\lambda_O^+ - \lambda_O^-) = \sqrt{c^2+4D_2 \alpha_O} \ .
\end{align*}
For any positive choices of $\alpha_B, \alpha_O, \alpha_S$ and $\alpha_I$, there exists a positive value $\mu_0$ such that if 
\begin{equation*}
    \mu_0^+ = \mu_0\ , \quad\quad   \mu_0^- = -\mu_0\ ,
\end{equation*}
we have the inequalities
\begin{eqnarray*}
   \max \left(\lambda_B^-,\lambda_O^- ,\frac{-\alpha_S}{c},\frac{-\alpha_I}{c} \right) < &\mu^{-}_0& < 0   \ ,\ \  \\
    0 < &\mu^+_0& < \min \left(\lambda_B^+,\lambda_O^+ ,\frac{\alpha_S}{c},\frac{\alpha_I}{c} \right) \ .
\end{eqnarray*}
Then for functions $B(z),O(z),S(z)$ and $I(z)$ in the space 
\begin{equation*}
    C_{\mu^-_0,\mu^+_0}(\mathbb{R}) := \{ h(z) \in C(\mathbb{R} ) \ \ | \ \sup_{z \leq 0} |h(z)e^{-\mu^-_0 z}| +\sup_{z \geq 0} |h(z)e^{-\mu^+_0 z}| < \infty \} \ ,
\end{equation*}
the inverse operators for each component of  $\Delta$ have the following integral representation:
\begin{align*}
    \Delta_B ^{-1} (B) (z) :=& \frac{1}{\rho_B} \left (
     \int_{-\infty}^z e^{\lambda_B^- (z - y)}B(y)dy + \int_z^\infty e^{\lambda_B^+ (z - y)}B(y)dy
     \right ), \\
     \Delta_O^{-1} (O) (z) :=& \frac{1}{\rho_O} \left (
     \int_{-\infty}^z e^{\lambda_O^- (z - y)}O(y)dy + \int_z^\infty e^{\lambda_O^+ (z - y)}O(y)dy
     \right ),  \\
     \Delta_S^{-1}(S)(z) :=& \frac{1}{c} \left (
     \int_z^{\infty} e^{ \frac{\alpha_S}{c}(z - y)}S(y)dy \right ), \\
     \Delta_{I}^{-1}(I)(z) :=& \frac{1}{c} \left (
     \int_z^{\infty} e^{ \frac{\alpha_{I}}{c}(z - y)}I(y)dy \right  ) \ .
\end{align*} \medskip
To see this for $\Delta_S^{-1}$ and $\Delta_B^{-1}$, first take the derivatives of each operator with respect to $z$,
\begin{align*}
   (\Delta_S^{-1} S)'(z) =& -\frac{S(z)}{c} + \frac{\alpha_S}{c^2}  \int_z^{\infty}  e^{ \frac{\alpha_S}{c}(z - y)}  S(y)dy,  \\
   (\Delta_B^{-1} B)'(z) =& \frac{\lambda_B^-}{\rho_B} \left (
     \int_{-\infty}^z e^{\lambda_B^- (z - y)}B(y)dy \right ) +\frac{\lambda_B^+}{\rho_B} \left ( \int_z^\infty e^{\lambda_B^+ (z - y)}B(y)dy
     \right ), \\
      (\Delta_B^{-1} B)''(z) =& \frac{(\lambda_B^-)^2}{\rho_B} \left (
     \int_{-\infty}^z e^{\lambda_B^- (z - y)}B(y)dy \right ) +\frac{(\lambda_B^+)^2}{\rho_B} \left ( \int_z^\infty e^{\lambda_B^+ (z - y)}B(y)dy
     \right ) - \frac{B(z)}{D_1} \ .
\end{align*}
Then it follows that
\begin{align*}
\Delta_S (\Delta_S^{-1} S) (z) & = 
- c (\Delta_S^{-1} S)' (z) + \alpha_S (\Delta_S^{-1}) S(z)\\
&= c \left ( \frac{S(z)}{c} - \frac{\alpha_S }{c^2}  \int_z^{\infty}  e^{ \frac{\alpha_S}{c}(z - y)}  S(y)dy  \right )  \\
    & + \alpha_S \left (  \frac{1}{c}\int_z^{\infty} e^{\frac{\alpha_S}{c}(z-y)}S(y)dy     \right )\\
    &=S(z) \ .
\end{align*}
Similarly, one can show for $B$ that
\begin{eqnarray*}
\Delta_B (\Delta_B^{-1} B) (z) &= & -D_1(\Delta_B^{-1} B)'' (z) - c (\Delta_B^{-1} B)' (z)+ \alpha_B (\Delta_B^{-1} B) (z) \\
     &=& \frac{-D_1 (\lambda_B^-)^2 - c\lambda_B^- +\alpha_B}{\rho_B} \left (  \int_{-\infty}^z e^{\lambda_B^- (z - y)}B(y)dy  \right ) \\ && + \frac{-D_1 (\lambda_B^+)^2 - c\lambda_B^+ +\alpha_B}{\rho_B} \left (  \int_z^\infty e^{\lambda_B^+ (z - y)}B(y)dy  \right ) + B(z) \\
     &= &B(z) \ .
\end{eqnarray*}
Here the last line follows from $\lambda_B^{\pm}$ being the characteristic roots of $\Delta_B$. It is easy to see these relations hold for $\Delta_I^{-1}$ and $\Delta_O^{-1}$ as the calculations are identical to those of $\Delta_S^{-1}$ and $\Delta_B^{-1}$ .\bigskip

We now define the Banach space
\begin{equation*}
    \mathbb{B}_{\mu_0}(\mathbb{R},\mathbb{R}^4) :=  C_{\mu^-_0,\mu^+_0} (\mathbb{R}) \times  C_{\mu^-_0,\mu^+_0} (\mathbb{R}) \times  C_{\mu^-_0,\mu^+_0} (\mathbb{R}) \times  C_{\mu^-_0,\mu^+_0} (\mathbb{R})  \ ,
\end{equation*}
with the norm
\begin{equation*}
    \|u\|_{\mu_0} := \max_{1\leq i \leq 4} \sup_{z \in \mathbb{R}} \{e^{-{\mu_0} |z|}|u_i(z)| \} \ ,
\end{equation*}
where $u:=(u_1,u_2,u_3,u_4) \in \mathbb{B}_{\mu_0}(\mathbb{R},\mathbb{R}^4)$, and on this function space we define the integral mapping
\begin{equation*}
    F(u):= \begin{bmatrix} F_1(u) \\ F_2(u) \\ F_3(u) \\ F_4(u)  \end{bmatrix} = \begin{bmatrix} 
    \Delta_B^{-1} \left(\alpha_B u_1+ ru_1\left (1-\frac{u_1}{K(u_3+u_4) + \epsilon}\right ) + \mu u_2 u_3 \right) \\
    \Delta_O^{-1} (\alpha_O u_2 + \alpha u_4 -\mu u_2 u_3 - \gamma u_2 ) \\
    \Delta_S^{-1} (\alpha_S u_3 - f(u_1)u_3) \\
    \Delta_I^{-1} (\alpha_I u_4 + f(u_1)u_3 - g(u_4)u_4 )
    \end{bmatrix} \ .
\end{equation*}

We are in a position now to state the main result of this subsection.

\begin{lemma} \label{lemma:fixedpoint}
Assume the mapping $F(u)$ admits a fixed point, $F(u)=u$, for $u=(B,O,S,I) \in \mathbb{B}_{\mu_0}(\mathbb{R},\mathbb{R}^4)$. Then $(B,O,S,I)$ satisfy the travelling wave equations (\ref{B:tweq})-(\ref{I:tweq}).
\end{lemma}

\begin{proof}
If the mapping $F(u)$ achieves a fixed point $u=(B,O,S,I)$ such that $F(u)=u$, then we have that 
\begin{equation*}
    F_3(u)= \Delta_S^{-1} \left ( \alpha_S S - f(B)S \right)=   S \ .
\end{equation*}
By applying $\Delta_S$, we find
$$\Delta_S \Delta_S^{-1} (\alpha_S S - f(B) S) = \Delta_S S = - c S' + \alpha_S S \ ,$$
and then 
\begin{equation*}
   - cS' = -f(B)S\ .
\end{equation*}
 This is the original differential equation for $S$ as seen in \eqref{S:tweq}, proving that this fixed point is equivalent to a solution for the $S$-component. One can directly show that this also holds for the other components of the map $F(u)$.
\end{proof}

Finally, we define the set
\begin{equation*}
    \Gamma := \{ (B,O,S,I) \in \mathbb{B}_{\mu_0^-,\mu_0^+}(\mathbb{R},\mathbb{R}^4) \ | \ \underline{B} \leq B \leq \overline{B}, \  \underline{O} \leq O \leq \overline{O}, \  \underline{S} \leq S \leq \overline{S}, \  \underline{I} \leq I \leq \overline{I}\ \} \ .
\end{equation*}

\subsubsection{$F(u)$ is invariant, continuous, and compact on $\Gamma$}

We choose $\alpha_B$, $\alpha_O$, $\alpha_S$ and $\alpha_I$ large enough such that
\begin{align*}
 &\frac{\partial }{\partial B} \Big( \alpha_B B(z) + rB(z) \big(1-\frac{B(z)}{K(S(z)+I(z))+\epsilon} \big) + \mu \cdot O(z)\cdot S(z) \Bigr)&> 0 \ , \ \forall \ z \in \mathbb{R} \ ,\\
 &\frac{\partial }{\partial O} \Big( \alpha_O O(z)+ \alpha \cdot I(z) - \mu \cdot O(z)\cdot S(z) - \gamma O(z) \Bigr) &> 0 \ , \ \forall \ z \in \mathbb{R} \ ,  \\
 &\frac{\partial }{\partial S} \Big(\alpha_S S(z)-f(B(z))S(z)  \Bigr) &> 0 \ , \ \forall \ z \in \mathbb{R} \ ,  \\
 &\frac{\partial }{\partial I} \Big( \alpha_I I(z) +  f(B(z))S(z) - g(I(z))I(z)\Bigr) &> 0 \ , \ \forall \ z \in \mathbb{R} \ .
\end{align*}
We have then the following result.
\begin{lemma}\label{lemma:Invariant} For any $u = (B,O,S,I) \in \Gamma$ ,
\begin{align*}
    & \underline{B} \leq F_1(u) \leq \overline{B}, \\
    & \underline{O} \leq F_2(u) \leq \overline{O}, \\
    & \underline{S} \leq F_3(u) \leq \overline{S}, \\
    & \underline{I} \leq F_4(u) \leq \overline{I} \ .
\end{align*}
\end{lemma}

\begin{proof}
With the above choices of $\alpha_{B,O,S,I}$, 
in combination with the inequalities of the upper and lower solutions of Section~\ref{UpLoBounds:sec}, we find
\begin{align*}
    \Delta_B \underline{B}(z) &= -D_1 \underline{B}''(z) - c\underline{B}'(z) +\alpha_B \underline{B}(z) \\ 
    &\leq  r\underline{B}(z)\left(1-\frac{\underline{B}(z)}{K(\underline{S}(z)+\underline{I}(z))+\epsilon}\right) + \mu \cdot \underline{O}(z)\cdot \underline{S}(z)+ \alpha_B \underline{B}(z) \\
    &\leq rB(z)\left(1-\frac{B(z)}{K(\underline{S}(z)+\underline{I}(z))+\epsilon}\right) + \mu \cdot \underline{O}(z)\cdot \underline{S}(z)+ \alpha_B B(z) \\
    &\leq rB(z)\left(1-\frac{B(z)}{K(S(z)+I(z))+\epsilon}\right) + \mu \cdot O(z)\cdot S(z)+ \alpha_B B(z) \\
    &\leq r\overline{B}(z)\left(1-\frac{\overline{B}(z)}{K(S(z)+I(z))+\epsilon}\right) + \mu \cdot O(z)\cdot S(z)+ \alpha_B \overline{B}(z) \\
    &\leq r\overline{B}(z)\left(1-\frac{\overline{B}(z)}{K(\overline{S}(z)+\overline{I}(z))+\epsilon}\right) + \mu \cdot \overline{O}(z)\cdot \overline{S}(z)+ \alpha_B \overline{B}(z) \\
    &\leq -D_1 \overline{B}''(z) - c\overline{B}'(z) + \alpha_B \overline{B}(z) \\
    &= \Delta_B \overline{B}(z) \ .
\end{align*}
Then, by the continuity of the upper and lower solutions, and through the application of Lemma 2.1 in \citep{shu2019traveling}, we get, for all $z$,
\begin{align*}
    \underline{B} (z) & \leq \Delta_B^{-1}( \Delta_B \underline{B} )(z) \\ &\leq \Delta_B^{-1} \left( rB \Big(1-\frac{B}{K(S+I)+\epsilon}\Bigr) + \mu \cdot O\cdot S+ \alpha_B B  \right ) (z)\\  &= F_1(u) (z)\\ &\leq \Delta_B^{-1}( \Delta_B \overline{B} )(z) \\  & \leq  \overline{B} (z)\ .
\end{align*}
The inequalities for $F_2(u), F_3(u)$ and $F_4(u)$ are proved in precisely the same way.
\end{proof}

Up to this point, we have constructed a convex Banach space $\Gamma$, an integral map $F$, and have shown that for all $u \in \Gamma$, $F(u) \in \Gamma$. In order to apply Schauder's fixed point theorem, we need to prove that $F$ is continuous and compact on $\Gamma$. These properties are proved below, and rely heavily on the fact that each function in $\Gamma $ is bounded above by some positive constant.
\begin{lemma} \label{lemma: Cont}
The mapping $F$ is continuous under the norm $\|\ . \|_{\mu_0}$ in $\Gamma$.
\end{lemma}
\begin{proof}
We have for each $u=(u_1,u_2,u_3,u_4) \in \Gamma$ , $z \in \mathbb{R}$,
\begin{align*}
    0 \leq \underline{B} (z) \leq u_1(z) &\leq \overline{B}(z) \leq B_0 \ , \\
    0 = \underline{O} (z)\leq u_2(z) &\leq \overline{O}(z) \leq \frac{\alpha I_0}{\gamma} \ , \\
    0 \leq \underline{S} (z)\leq u_3(z) &\leq \overline{S}(z) =  N  \ , \\
     0 = \underline{I} (z)\leq u_4(z) &\leq \overline{I}(z) =  I_0 \ .
\end{align*}
We borrow the continuity proof of \citet{shu2019traveling} (Lemma 4.1) in order to prove continuity for the map components $F_1$ and $F_2$. Here, they make the assumption that the gradient of their reaction terms are uniformly bounded. Since any function in~$\Gamma$ is uniformly bounded and our reaction terms are~$C^1$, one may check that we too satisfy this condition. This allows us to claim that there exists a constant~$L$ and a function $g$ such that for any $u,v \in \Gamma$,

\begin{equation*}
    \left |F_1(u)(z) - F_1(v)(z) \right | e^{-\mu_0|z|} \leq L g(z) \|u-v\|_{\mu_0} \ \  ,
\end{equation*}
where 
\begin{equation*}
    g(z) = e^{-\mu_0 |z|} \left [\int_{-\infty}^z e^{\lambda_B^-(z-y)+\mu_0 |y|} dy  +  \int_z^\infty e^{\lambda_B^+(z-y)+\mu_0 |y|} dy\right ] \ .
\end{equation*}
Since $\lambda_B^- < -\mu_0 <\mu_0 < \lambda_B^+$, one can apply L'Hopital's rule to show that $g$ is uniformly bounded on $\mathbb{R}$. Then there exists a positive constant $C_1$ such that, for any $u,v\in \Gamma$,
\begin{equation*}
       \|F_1(u) -F_1(v) \|_{\mu_0} \leq C_1 \cdot \|u-v\|_{\mu_0} \ .
\end{equation*}
and in particular $F_1$ is continuous in $\Gamma$. The proof is identical for showing continuity of the component $F_2(u)$. \medskip

To see that the ODE components $F_3(u)$ and $F_4(u)$ are continuous, observe that for any $u:=(u_1,u_2,u_3,u_4),v:=(v_1,v_2,v_3,v_4) \in \Gamma $, we have 

\begin{equation*}
    |F_3(u)(z)-F_3(v)(z)| = \frac{1}{c}\left | \int_z^{\infty} e^{\frac{\alpha_S}{c}(z-y)} \left (\alpha_S u_3 -f(u_1)u_3 -\alpha_S v_3 +f(v_1)v_3 \right)dy \right | \ .
\end{equation*}
Here we can use the Lipschitz property of the Hill function $f$,
\begin{equation*}
    |f(B_1)-f(B_2)| \leq K_f |B_1 - B_2| \ ,
\end{equation*}
and, by adding and subtracting $f(v_1)u_3$ to the equation,
\begin{equation*}
    |F_3(u)(z)-F_3(v)(z)| = \frac{1}{c}\left | \int_z^{\infty} e^{\frac{\alpha_S}{c}(z-y)} \left (\alpha_S u_3 -f(u_1)u_3 +f(v_1)u_3 - f(v_1)u_3-\alpha_S v_3 +f(v_1)v_3 \right)dy \right | \ .
\end{equation*}
This allows us to re-arrange some of the terms and find
\begin{equation*}
    |F_3(u)(z)-F_3(v)(z)| \leq \frac{1}{c}\left | \int_z^{\infty} e^{\frac{\alpha_S}{c}(z-y)} \left (\alpha_S |u_3-v_3| + N K_f |u_1 - v_1| +M_1 |u_3-v_3| \right)dy \right | \ .
\end{equation*}
Multiplying both sides by $e^{-\mu_0|z|}$ gives us
\begin{align*}
    |F_3(u)(z)-F_3(v)(z)|e^{-\mu_0|z|} &\leq \frac{1}{c} \bigg | \int_z^{\infty} e^{\frac{\alpha_S}{c}(z-y)}  (\alpha_S |u_3-v_3|e^{-\mu_0|z|}  \\
    &+ N K_f |u_1 - v_1|e^{-\mu_0|z|} +M_1 |u_3-v_3|e^{-\mu_0|z|} )dy \bigg | \ .
\end{align*}
One may then see that
\begin{eqnarray*}
  &&  |F_3(u)(z)-F_3(v)(z)|e^{-\mu_0|z|} \\
  &\leq & \frac{1}{c} \left | \int_z^{\infty} e^{\frac{\alpha_S}{c}(z-y)} e^{\mu_0 (|y|-|z|)} \left( \alpha_S \|u-v\|_{\mu_0} + N K_f \|u-v\|_{\mu_0} +M_1 \|u-v\|_{\mu_0} \right)dy \right | \\
    &= & \frac{1}{c} \|u-v\|_{\mu_0} (\alpha_S +NK_f +M_1) \int_z^{\infty} e^{\frac{\alpha_S}{c}(z-y)} e^{\mu_0 (|y|-|z|)}  dy \\
    &= & C_3 \cdot \|u-v\|_{\mu_0} \ ,
\end{eqnarray*}
for some $C_3 >0$ and for all $z \in \mathbb{R}$.

It is straightforward to prove that, for $F_4$, we also reach a result of the form 
\begin{equation*}
    \sup_{z\in \mathbb{R} }\{|F_4(u)(z)-F_4(v)(z)|e^{-\mu_0|z|}\} \leq C_4 \cdot \|u-v\|_{\mu_0} \ ,
\end{equation*}
where $C_4$ is some positive constant. Then, we have that there exists a constant~$C$ such that 
\begin{equation*}
    \max_{1\leq i \leq 4} \sup_{z\in \mathbb{R} }\{|F_i(u)(z)-F_i(v)(z)|e^{-\mu_0|z|}\} = \|F(u)-F(v)\|_{\mu_0} \leq C\cdot \|u-v\|_{\mu_0}  \ ,
\end{equation*}
for all $u,v\in \Gamma$. This completes the proof of Lemma~\ref{lemma: Cont}. 
\end{proof}

\begin{lemma}\label{lemma: comp}
The mapping $F$ is compact under the norm $\|\ . \ \|_{\mu_0}$ in $\Gamma$.
\end{lemma}

\begin{proof}
To show compactness, we demonstrate that for any sequence of functions $\{ u^n \} \in \Gamma$, where $n \in \mathbb{N}$, the sequence $\{F(u^n)\}$ has a convergent sub-sequence in $\Gamma$ with respect to the norm $\|\ . \ \|_{\mu_0}$. To do this, we use the Arzela-Ascoli theorem.

By the invariance of $\Gamma$, the sequence $\{F(u^n)\}$ is uniformly bounded with respect to the $L^\infty$-norm. Observe also that for any $u:= \left ( B,O,S,I \right )\in \Gamma$,
\begin{eqnarray*}
    && | F_1(u)'(z) | \\
    &=& \bigg | \frac{\lambda_B^-}{\rho_B} \left (
     \int_{-\infty}^z e^{\lambda_B^- (z - y)}\left(\alpha_B B(y)+ rB(y)\left (1-\frac{B(y)}{K(S(y)+I(y)) + \epsilon}\right ) + \mu O(y) S(y) \right) dy \right ) \\
     &+& \frac{\lambda_B^+}{\rho_B} \left ( \int_z^\infty e^{\lambda_B^+ (z - y)}\left(\alpha_B B(y)+ rB(y)\left (1-\frac{B(y)}{K(S(y)+I(y)) + \epsilon}\right ) + \mu O(y) S(y) \right)dy
     \right )  \bigg | \ .
\end{eqnarray*}
Due again to the boundedness of $\Gamma$ with respect to the $L^\infty$-norm, one can find a constant $W_1\in \mathbb{R}^+$ such that for all $z \in \mathbb{R}$,
\begin{align*}
    | F_1(u)'(z) | &\leq W_1 \ .
\end{align*}
One may similarly check that $F_2 (u)'$, $F_3 (u)'$ and $F_4 (u)'$ are uniformly bounded.






Now we know the sequence $\{ F (u^n)\}$ is uniformly bounded and equi-continous. By the Arzela-Ascoli theorem and a standard diagonal process, there exists a subsequence which converges locally uniformly. One may check that the convergence also occurs with respect to the norm $\|\ . \ \|_{\mu_0}$. We defer the details of this proof to Lemma 3.5 in \citet{shu2019traveling} or the appendix of \citet{wang2016traveling}.  This completes the proof of Lemma \ref{lemma: comp}.
\end{proof}

We now have by Lemma \ref{lemma:Invariant}, Lemma \ref{lemma: Cont} and Lemma \ref{lemma: comp} that the mapping~$F$ is invariant, continuous and compact on $\Gamma$ with respect to the norm $\|\ . \ \|_{\mu_0}$. By Schauder's fixed point theorem, $F$ admits at least one fixed point~$(B,O,S,I)$ in~$\Gamma$. By Lemma~\ref{lemma:fixedpoint}, this fixed point is a solution to the travelling wave system~\eqref{B:tweq}-\eqref{I:tweq}.

\subsubsection{Travelling waves asymptotics}

This leads us to our final step, which is to show that our solution satisfies the desired asymptotic conditions.

As $z\longrightarrow +\infty$, applying the squeeze theorem immediately yields
\begin{align*}
    B(z) \longrightarrow 0 \ \ &\text{as} \ \ z \longrightarrow +\infty\ , \\
    O(z) \longrightarrow 0 \ \ &\text{as} \ \ z \longrightarrow +\infty\ , \\
    S(z) \longrightarrow N \ \ &\text{as} \ \ z \longrightarrow +\infty\ , \\ 
    I(z) \longrightarrow 0 \ \ &\text{as} \ \ z \longrightarrow +\infty \ .
\end{align*}

Let us now look at $z \to -\infty$. There the upper and lower solutions have differing asymptotics and more work is necessary.
\begin{lemma}\label{lem:Binfty}
Let $(B,O,S,I) \in \Gamma$ be a fixed point of $F$ as constructed in the previous subsection. Then
$$\liminf_{z \to -\infty} B(z) >0.$$
\end{lemma}
\begin{proof}
    First, due to $O,S,I \geq 0$, we have
    
    $$-D_1 B'' - c B' \geq r B \left(1 - \frac{B}{\epsilon} \right)  \ .$$
    
   By construction, we know that $B \geq \underline{B}$ with $\underline{B}$ non-negative and not identically equal to 0. It follows from the strong maximum principle that $B >0$ on the whole line.    Moreover, recall that 

    $$\underline{B} (z)  = \max \left\{ 0 , e^{-\lambda z}- \nu e^{-(\lambda + \delta z)} \right\} \ , $$
    which, by the same computation as in Section~\ref{UpLoBounds:sec}, satisfies
    $$- D_1 \underline{B} '' - c\underline{B} \leq r \underline{B} \left( 1 - \frac{\underline{B}}{\epsilon} \right) \ .$$
    The same inequality holds for any spatial shift of $\underline{B}$.

    Now let us proceed by contradiction and assume that 
    
    $$\liminf_{z \to -\infty} B (z) = 0 \ . $$
    
    Then there exists some critical shift 
    
    $$\kappa^* := \sup \, \{ \kappa  \geq 0 \ | \ \ B(\cdot) \geq \underline{B}(\cdot + \kappa ) \} \in [0, +\infty) \ . $$
    
    Due to $\underline{B}$ being decreasing on a right half-line, one may check that there exists some $z^*$ such that $\underline{B} (z^*+\kappa^*) = B (z^*)$. By the strong maximum principle, we conclude that $\underline{B}(\cdot+\kappa^*)\equiv B(\cdot)$, a contradiction. 
    \end{proof}
With this lemma in hand, by \eqref{S:tweq}, we get

\begin{equation*}
    - c S'(z) = -f (B (z)) S(z) \leq - \delta_1 S(z) \,
\end{equation*}
for all $z<0$ and some $\delta_1>0$. By Gronwall's lemma, it follows that
$$S(z) \leq S(0) e^{\frac{\delta}{c}z}\ \ \text{for all } z \leq 0 \ ,$$
hence $S(-\infty)=0$.

Now, for any $\delta_2 >0$, there exists $z_0 >0$ large enough so that
$$\sup_{z \leq -z_0} S(z) \leq \delta_2\ ,$$
and plugging this new information into by \eqref{I:tweq}, we get that
$$-c I'(z) \leq M_1 \delta_2 - g(I) I \ \ \text{for all } z \leq -z_0 \ .$$

It follows that 
$$\limsup_{z \to -\infty} I (z) \leq \beta \ ,$$
where $\beta$ is the unique positive solution to $g(\beta) \beta  = M_1 \delta_2$. Up to reducing $\delta_2$, this~$\beta$ can be made arbitrarily small. By non-negativity of~$I$, we deduce that $I(-\infty)=0$.

Next, up to increasing $z_0$, by \eqref{B:tweq} we have that
$$r B \left( 1 - \frac{B}{\epsilon} \right)  \leq -D_1 B'' (z) - c B' (z) \leq  r B \left( 1 - \frac{B}{2K\delta_2 + \epsilon} \right) + \mu \frac{\alpha I_0}{\gamma} \delta_2 \ ,$$
for $z \leq -z_0$. 

 Assume first that $B$ is monotonic in a neighborhood of $-\infty$. Then, due to $B$ being bounded and by standard regularity estimates, we find that $B' , B'' \to 0$ and $B \to B_\infty $ as $z \to +\infty$. Moreover $B_\infty >0$ by Lemma~\ref{lem:Binfty}. Passing to the limit in the above inequalities, we get
$$r  B_\infty  \left( 1 - \frac{B_\infty}{\epsilon} \right) \leq 0 \  \leq r B_\infty \left( 1 - \frac{B_\infty}{2K\delta_2 + \epsilon} \right) + \mu \frac{\alpha I_0}{\gamma} \delta_2 \ . $$
Since $\delta_2$ can be arbitrarily small, we conclude that $B_\infty = \lim_{z \to -\infty} B (z) = \epsilon $.

On the other hand, if $B$ is not monotonic, then there exist two sequences $z_n, z_n' \to -\infty$ such that
$B' (z_n) = B' (z_n ' ) = 0$, $B'' (z_n) \geq 0 \geq B'' (z_n ')$ and $\lim_n B(z_n) = \liminf_{z \to -\infty} B (z)$, $\lim_n B (z_n ') = \limsup_{z \to -\infty} B (z)$. Evaluating the above inequality at $z_n$ and passing to the limit, we find that
$$r \liminf_{z \to -\infty} B (z) \left(1 - \frac{\liminf_{z \to -\infty} B(z) }{\epsilon} \right)  \leq 0\ .$$
By Lemma~\ref{lem:Binfty}, necessarily $\liminf_{z \to -\infty} B (z) \geq \epsilon$. Evaluating at $z_n '$, one gets an opposite inequality and eventually reaches again the conclusion that $B(z) \to \epsilon$ as $z \to -\infty$.

The fact that $O(z) \to 0$ as $z \to -\infty$ follows a similar argument and we omit the details. This completes our proof of existence of travelling waves, which we recap in Theorem~\ref{Thm3} in the next subsection. 

\subsubsection{Existence theorem for travelling waves}\label{twexistence:sec}

We are now in a position to state our main theorem on the existence of travelling waves, which we have just proved in the previous subsections.
\begin{theorem}\label{Thm3}
Assume that
$$D_1 \geq D_2 \ , \quad n_1 = 2 \ . $$
Then:
\begin{enumerate}[label=(\roman*)]
\item there exists no travelling wave solution to system~\eqref{B:tweq}-\eqref{I:tweq} with asymptotic conditions~\eqref{asymp+}-\eqref{asymp-} for any $c < 2 \sqrt{D_1 r}$;
\item if furthermore \eqref{main_mu_assumption} holds, then for any $c > 2 \sqrt{D_1 r}$, there exists a travelling wave solution to system~\eqref{B:tweq}-\eqref{I:tweq} with asymptotic conditions~\eqref{asymp+}-\eqref{asymp-};
\end{enumerate}
\end{theorem}
As we will discuss in the next section, this theorem provides a sufficient condition for the existence of travelling waves. We conjecture it is not necessary. In particular, when condition~\eqref{main_mu_assumption} is not satisfied, we expect that the minimal wave speed may be strictly larger than $2 \sqrt{D_1 r}$.

We further point out that the upper and lower solutions used in our existence proofs also provide an invariant region for the full evolution system. Indeed, for initial conditions $u_0 :=\left(B_0 ,O_0,I_0,S_0 \right )\in \Gamma$ the associated solution~$u (x,t) := \left( B,O,S,I \right) (x,t)$ of \eqref{B:pde}-\eqref{I:ode} also satisfies that
$$x \mapsto u (x - ct,t) \in \Gamma \ , $$
for any $t>0$. For the sake of expediency, we omit the proof but it relies on a parabolic comparison principle closely connected to the invariance of $\Gamma$ through the mapping~$F$ (see also Lemma~\ref{lemma:Invariant} above). We refer to~\citet{bhkr2005,giletti2010} for similar arguments. In particular, it follows that for any such initial data, the solution of the Cauchy problem spreads with the speed~$c$ of the corresponding travelling wave.

\section{Discussion}

\subsection{Modelling fire blight}

Fire blight is a complex, sporadic and destructive bacterial disease. The severity of epidemics appears to be increasing \citep{van2012losses} and, because of their overuse, antibiotic-based control methods may not always be effective \citep{thomson1992presence}. A further understanding of how the fire blight pathogen, insect vectors and hosts interact to cause new infections is critical if we are to prevent future epidemics.

Despite the long history of fire blight as a plant disease, its routine destruction of orchards and the large amount of biological literature available to study fire blight, it remains a neglected topic in the mathematical modelling community. Modelling the spread of plant diseases can help uncover how certain ecological and epidemiological processes interact to induce or accelerate epidemics. They can also be used to predict the spread of disease given a state of initial conditions.

Previous disease models of fire blight spread \citep{chen2018sliding,iljon2012mathematical} focused on the time evolution of disease severity. These models took the form of ordinary differential equations and described the transition of a host  or vector from a susceptible state to a disease-carrying one. Here, we chose to consider the evolution of the disease over both time and space. We derived a system of ordinary differential equations, to describe the stationary host population, and coupled them to a system of partial differential equations, to describe the vectored pathogen population. This allows us to investigate how the disease spreads across an orchard. Indeed, in Theorem \ref{Thm3}, we proved that, depending on parameters, our fire blight model permits travelling waves, i.e. disease propagation at constant speed.

One aspect of fire blight epidemiology that was explicitly considered in our model but not in the previous work is the production and dissemination of ooze by infected hosts. The inclusion of this symptom led to a second partial differential equation in our model. Ooze is known to be a major source of the pathogen population for new infections, but since Hildebrand found that bees were not attracted to ooze  \citep{hildebrand1936honeybee}, there has not been a lot of research into the production and dissemination of ooze apart from the works of \citep{slack2017microbiological, boucher2019effects}. Further research into how ooze is produced and spread may be key in preventing future fire blight epidemics.  In this context we should re-iterate that a base assumption in our model simulations and analysis is that the managed pollinators, primarily honeybees, are more efficient in visiting flowers than other insects that are responsible the spread of ooze. Our travelling wave analysis made use of this assumption in several places. We remark here that this assumption can be somewhat relaxed but then additional mathematical arguments are needed, which we did not present here.

Another aspect of fire blight modelling that previous spatially homogeneous works have not had to consider is the question of boundary conditions. We assumed in our simulations and analysis of the model on a bounded domain (i.e. an orchard of finite size) that the in-flux and out-flux of bacteria due to pollinators could be considered negligible. However, for previously blight-free orchards, the pathogen must enter from an outside source for an epidemic to be possible. Robin type boundary conditions may prove to be the most accurate representation of reality for fire blight, where there is an initial pathogen population existing at the edge of an orchard that is transported into the orchard at some rate. More data on the spatial spread of fire blight is needed in order to determine which boundary conditions are most appropriate. 

\subsection{Ecological interpretation of constraint on the ooze conversion rate $\mu$}

The results of Theorem \ref{Thm3} offer some biological interpretations. When the diffusion of the floral bacteria (due to bees) is greater than the diffusion of the bacteria within ooze (due to non-bee insects), and provided that the ooze conversion is slow enough, then pathogen can invade the host population in the form of a travelling wave whose speed can be explicitely characterized as $2 \sqrt{D_1 r}$, i.e. a function of the pollinators' motion and pathogen's growth on flowers.

Our travelling wave proof hinges on the dynamics of the bacteria population~$B$ being the controlling influence, which is described by a Fisher-like equation, albeit with some important differences: most importantly the carrying capacity is not constant but depends on some of the state variables. A direct consequence of this is that the wave for $B$ is not monotonic but pulse like.

In order to achieve this dominance of the dynamics of $B$ it was important to bound the rate $\mu$ of conversion of ooze $O$ to free bacteria $B$, which is a source term in the equation for $B$. When~$\mu$ is too large, it remains an open question whether travelling waves still exist (though numerical simulations suggest that they do), but our Theorem~\ref{Thm3} states that any such travelling wave should still be faster than $2 \sqrt{D_1 r}$. In that sense, our existence result provides a sufficient condition for the speed of the fire blight disease to be equal to the smallest value possible, i.e. $2 \sqrt{D_1 r}$.

More precisely, the requirement that $\mu$ is small enough so that (\ref{main_mu_assumption}) holds, is achieved if 
\begin{itemize}
\item the bacterial growth rate $r$ is large enough
\item the loss rate of spore viability $\gamma$ is large enough
\item the spore production rate $\alpha$ is small enough
\item the infection rate $M_1$ is small 
\item the infection threshold $A_1$ is large
\item the carrying capacity for $B$, expressed in terms of $K$ and $\epsilon$ is small enough
\end{itemize}
or any appropriate combination of these factors.


For eaxample, such travelling waves are guaranteed to form if the disease infection process proceeds slowly,   and if the blossom associated bacteria reproduce quickly while having a low carrying capacity, i.e. if they are $r$-strategists, rather than $K$-strategists.


Accordingly, a remedial strategy may be to focus on removing ooze (for the minimal wave speed to equal $2\sqrt{D_1 r}$) and keeping the bacterial reproduction rate low (to decrease the speed $2 \sqrt{D_1} r$).

\subsection{Additional observation from numerical simulations, conjectures and open questions}

\begin{figure}
\includegraphics[width=\textwidth]{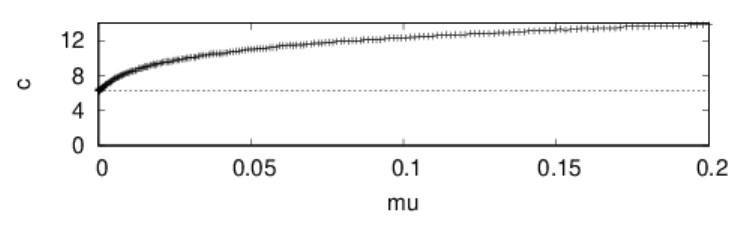}
\caption{Numerically observed wave speed $c$ for various values of spore liberation rate $\mu$. All other parameters were kept at their values from Table \ref{param_tab}. The value $2 \sqrt{D_1 r} =6.32$ of the minimal wave speed is indicated by the horizontal line. For these parameters Theorem~\ref{Thm3} guarantees travelling wave solutions for $\mu< 0.0004$.}\label{TWmu:fig}
\end{figure}

\begin{figure}
\begin{tabular}{cc}
\includegraphics[width=0.5\textwidth]{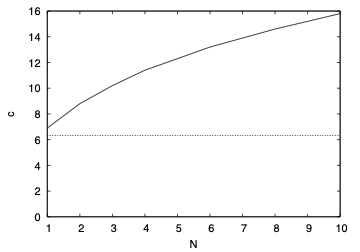}&
\includegraphics[width=0.5\textwidth]{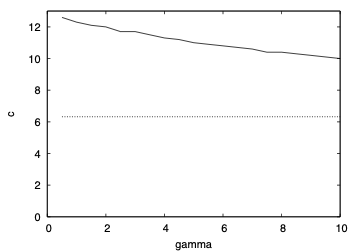}\\
(a) & (b)
\end{tabular}
\caption{Numerically observed wave speed $c$ for varying blossom density $N$ (left) and ooze removal rate $\gamma$ (right);
the minimal waves speed $2 \sqrt{D_1 r}=6.32$ is marked by the horizontal line;
in these simulations $\mu=0.1$, all other parameters are as in Table \ref{param_tab}.}\label{Ngamma:fig}
\end{figure}

Our main result in Theorem \ref{Thm3} is not optimal. It only gives sufficient conditions for a travelling wave to form. Numerical observations suggest that these are not necessary. This is for example also the case in the illustrative simulations that we showed above in Figure~\ref{illustTW}.
That the conditions are only sufficient stems partly from the fact that in the scenario that we studied the dynamics of the Fisher equation for $B$, albeit with density dependent carrying capacity, dominates the dynamics of the entire system.
We conjecture, therefore based on our numerical simulations,  that the travelling wave phenomenon occurs in a wider parameter range, and at wave speeds that can be substantially larger than $2 \sqrt{D_1 r}$, see for example, Figure \ref{TWmu:fig}.

As one might expect, higher spore liberation rates $\mu$ lead to higher flower associated pathogen densities, which in turn increases the number of infected hosts. Therefore, not only does the travelling wave move faster, but it has also more pronounced peaks in $B$ and $I$, if $\mu$ increases (data not shown).

From the perspective of pest management, one might expect that the density of hosts $N$ as well the spore decay rate $\gamma$ are key parameters.  We have investigated this in simple simulations which are reported in Figure \ref{Ngamma:fig}. We observe that increasing host density accelerates the spread of the disease across the orchard (see Fig. \ref{Ngamma:fig}.a).
This observation might in some sense confirm the results in \citep{iljon2012mathematical}. Whereas in that paper a spatially implicit model was studied that cannot predict the spatial spread rate, they showed that a greater number of hosts can result in an orchard that is more susceptible to fire blight infections. The notion that a greater host density increases the severity of disease spread seems consistent with other disease models and has some experimental evidence in ecology \citep{cunningham2021quantifying,lafferty2004fishing}. This finding would suggest that, as growers move towards higher-density plantings of susceptible cultivars, fire blight control strategies will need to be improved considerably. Otherwise, we may see more frequent and severe epidemics of the blight.

On the other hand, increasing $\gamma$ leads to slower wave speeds, see Fig. \ref{Ngamma:fig}.b. The parameter $\gamma$ accounts for the natural loss of pathogen viability, but would also account for other remedial strategies that are based on removing pathogens, for example removal of limbs, or bacteriocidal antibiotics that target the pathogens encased in ooze (bacteriocstatic antibiotics would be accounted for in the model by reducing parameters $\alpha$ and $r$).



These observations show that much more research on this model is required, both concerning its analysis, as well as numerical exploration of possible remedial strategies.

\subsection{Travelling waves in coupled ODE-PDE systems}

One prediction of our model is that the fire blight pathogen can invade an orchard in the form of a travelling wave. There are many reasons why it might be difficult to observe this phenomenon in nature. For one, travelling wave solutions are solutions on the unbounded domain, whereas realistic orchards are of finite size and the model will need to be subjected to boundary conditions that connect the domain with the outside world. Nevertheless, such travelling wave solutions offer insight into the spatio-temporal dynamics of a diffusion-reaction system, for example the spread of a disease and the time and length scale on which this happens.

Populations that invade new habitats often do so in the form of a travelling wave \citep{da1990spread, zadoks1994spread}, and experimental evidence exists for travelling waves in plant diseases \citep{buiel1989effect,minogue1983models,van1988focus}. If fire blight does in fact travel as a wave, then it may be possible to empirically determine a constant speed of spread and use this speed to predict where new infections are likely to occur.

Although ODE-PDE systems see less attention in the travelling wave literature than strictly parabolic systems, there are a number of approaches have been taken to prove that they admit travelling wave solutions. 

In \citet{li2012traveling}, the authors use a compact integral operator to prove that travelling wave solutions exist and they identify a minimum wave speed for the system. The approach taken here is very similar to theirs, except that their proof required the system to be cooperative and ours did not. However, their approach led to both sufficient and necessary conditions for which a travelling wave solution can exist, whereas our proof led only to sufficient conditions. 

Another approach to proving the existence of travelling waves in ODE-PDE systems is to use phase-plane techniques and linearize the system about its steady states to prove that one equilibrium is unstable while the other is stable. This is done, for example, in \citet{logan2008introduction}, where travelling wave solutions are shown to exist in an ODE-PDE system to model the spread of rabies in foxes. This same approach could not be applied here, as our linearization led to multiple zero eigenvalues giving inconclusive results. A similar approach in the higher dimensional case is to use shooting-methods as in \citet{dunbar1983travelling}. This is a difficult approach to take in high dimensional systems like ours, but it was successfully applied in \citet{ai2010traveling} to a model with the same dimension but milder nonlinearities in the reaction terms.  

To prove the existence of travelling wave solutions in our model, we used Schauder's fixed point theorem and the method of upper and lower solutions. The approach taken here is identical to the ones in \citet{ma2001traveling,huang2006travelling,shu2019traveling,shang2016traveling,zhou2019critical}: Define an integral mapping, show that a fixed point of this mapping is a travelling wave solution, show that the mapping does indeed admit a fixed point and that this solution has the desired asymptotic behaviour. The challenge of using this method here was that, unlike the examples listed above, some of our diffusion coefficients were zero. As a result, we had to define integral and differential operators for the ODE components of the system that we had not previously seen in the literature. We then showed that we could proceed with the travelling wave proof, as in the examples above, in precisely the same way but with these new operators. 

Our travelling wave proof required the construction of upper and lower solutions so that Schauder's fixed point theorem could be applied. Unfortunately, our choices of upper and lower solutions required a series of conditions on both, model parameters and the wave speed, to be met. Our simulation results suggest strongly that these restrictions are a product of our inability to find more suitable upper and lower solutions. It may be possible to find upper and lower solutions that do not impose these restrictions. Alternatively, one can try and apply the approach used in \citet{ai2010traveling} to remove the conditions, or, one can start with lower solutions that do not impose restrictions, take constant upper solutions and find a new way to prove that the asymptotic behaviour is satisfied.

\section{Conclusion}

In this work we formulated a model to describe how the fire blight pathogen, \textit{E.amylovora}, moved amongst a host population of blossoms during bloom. After listing the assumptions we made about the hosts, the pathogen and the vectors that disperse the pathogen, we formulated a model consisting of two reaction-diffusion partial differential equations and three ordinary differential equations. The two partial differential equations describe the population dynamics and the dispersal of the pathogen, whereas the ordinary differential equations describe the disease dynamics of our stationary host population. The model led to the following predictions and conclusions:

\begin{itemize}

    \item In our model, fire blight can move through an orchard in the form of a travelling wave, meaning that the pathogen population will invade the stationary host population with a constant speed and shape. The speed at which the pathogen invades the host population may be positively correlated with host density, the rate at which ooze is transferred onto healthy flowers and the average amount of ooze produced per infected individual. Where current fire blight control methods focus on the pathogen population existing on open blossoms, new methods that focus on controlling the production and dissemination of ooze may result in less fruit loss to fire blight. Further, the shift towards planting higher density orchards may result in mroe frequent and severe fire blight epidemics.

    \item  Travelling wave solutions can be shown to exist in high-dimensional non-linear ODE-PDE coupled systems through the application of Schauder's fixed point theorem.  We showed that by altering the differential operators and integral mapping used in \citep{ma2001traveling,wang2016traveling} to incorporate ordinary differential equations, one could use Schauder's fixed point theorem, together with the method of upper and lower solutions, to prove that travelling wave solutions exist in reaction-diffusion systems for which some of the diffusion coefficients are zero. Although we found this method useful, the method of upper and lower solutions introduced a series of restrictions on both model parameters and the wave speed that we believe to be sufficient and not necessary.

\end{itemize}

\subsection*{Acknowledgement}

M.P. was supported by the Ontario Ministry of Agriculture, Food and Rural Affairs (OMAFRA) and the Canada First Research Excellence Fund through an {\it Ontario Agrifood Innovation Alliance}/{\it Food from Thought} HQP Scholarship.

X.-S. W. is partially supported by the Louisiana Board of Regents Support Fund under contract No. LEQSF(2022-25)-RD-A-26.

H.J.E acknowledges financial support received from the Natural Sciences and Engineering Research Council of Canada's (NSERC) Discovery Grant program (grant no. RGPIN-2019-05003).

\bibliography{mybib}

\begin{thebibliography}{70}
\providecommand{\natexlab}[1]{#1}
\providecommand{\url}[1]{\texttt{#1}}
\expandafter\ifx\csname urlstyle\endcsname\relax
  \providecommand{\doi}[1]{doi: #1}\else
  \providecommand{\doi}{doi: \begingroup \urlstyle{rm}\Url}\fi

\bibitem[Ai(2010)]{ai2010traveling}
S.~Ai.
\newblock Traveling waves for a model of a fungal disease over a vineyard.
\newblock \emph{SIAM {J}ournal on {M}athematical {A}nalysis}, 42\penalty0 (2):\penalty0 833--856, 2010.

\bibitem[Andow' et~al.(1990)Andow', Kareiva, Levin, and Okubo]{da1990spread}
D.A. Andow', P.M. Kareiva, S.A. Levin, and A.~Okubo.
\newblock Spread of invading organisms.
\newblock \emph{Landscape {E}cology}, 4\penalty0 (2/3):\penalty0 177--188, 1990.

\bibitem[Aronson and Weinberger(1975)]{aronson1975}
D.~G. Aronson and H.~F. Weinberger.
\newblock Nonlinear diffusion in population genetics, combustion, and nerve pulse propagation.
\newblock {\it Partial {D}iffer. {Equat}. {R}elat. {Top}.}, {Tulane} {Univ}. 1974, {Lect}. {Notes} {Math}. 446, 5-49 (1975)., 1975.

\bibitem[Berestycki et~al.(2005)Berestycki, Hamel, Kiselev, and Ryzhik]{bhkr2005}
H.~Berestycki, F.~Hamel, A.~Kiselev, and Lenya Ryzhik.
\newblock Quenching and propagation in {KPP} reaction-diffusion equations with a heat loss.
\newblock \emph{Archive for {R}ational {M}echanics and {A}nalysis}, 178\penalty0 (1):\penalty0 57--80, 2005.
\newblock ISSN 0003-9527.
\newblock \doi{10.1007/s00205-005-0367-4}.

\bibitem[Billing(1979)]{billing1979warning}
E.~Billing.
\newblock Warning systems for fireblight 1.
\newblock \emph{EPPO Bulletin}, 9\penalty0 (1):\penalty0 45--51, 1979.

\bibitem[Boucher et~al.(2019)Boucher, Collins, Cox, and Loeb]{boucher2019effects}
M.~Boucher, R.~Collins, K.~Cox, and G.~Loeb.
\newblock Effects of exposure time and biological state on acquisition and accumulation of erwinia amylovora by drosophila melanogaster.
\newblock \emph{Applied and {E}nvironmental {M}icrobiology}, 85\penalty0 (15):\penalty0 e00726--19, 2019.

\bibitem[Britton(1991)]{Britton1991}
N.~F. Britton.
\newblock An integral for a reaction-diffusion system.
\newblock \emph{Appl. Math. Lett.}, 4\penalty0 (1):\penalty0 43--47, 1991.
\newblock ISSN 0893-9659.
\newblock \doi{10.1016/0893-9659(91)90120-K}.

\bibitem[Buiel et~al.(1989)Buiel, Verhaar, van~den Bosch, Hoogkamer, and Zadoks]{buiel1989effect}
A.A.M. Buiel, M.A. Verhaar, F.~van~den Bosch, W.~Hoogkamer, and J.C. Zadoks.
\newblock The effect of variety mixtures on the expansion velocity of yellow stripe rust foci in winter wheat.
\newblock \emph{Netherlands {J}ournal of {A}gricultural {S}cience}, 37:\penalty0 75--78, 1989.

\bibitem[Burrill(1880)]{burrill1880anthrax}
T.J. Burrill.
\newblock Anthrax of fruit trees: Or the so-called fire blight of pear, and twig blight of apple, trees.
\newblock \emph{Transactions of the {I}llinois {S}tate {H}orticultural {S}ociety}, pages 157--167, 1880.

\bibitem[Chen et~al.(2018)Chen, Kang, et~al.]{chen2018sliding}
C.~Chen, Y.~Kang, et~al.
\newblock Sliding motion and global dynamics of a {F}ilippov fire-blight model with economic thresholds.
\newblock \emph{Nonlinear Analysis: Real World Applications}, 39:\penalty0 492--519, 2018.

\bibitem[Cunningham et~al.(2021)Cunningham, Comte, McCallum, Hamilton, Hamede, Storfer, Hollings, Ruiz-Aravena, Kerlin, Brook, et~al.]{cunningham2021quantifying}
C.X. Cunningham, S.~Comte, H.~McCallum, D.G. Hamilton, R.~Hamede, A.~Storfer, T.~Hollings, M.~Ruiz-Aravena, D.H. Kerlin, B.W. Brook, et~al.
\newblock Quantifying 25 years of disease-caused declines in tasmanian devil populations: host density drives spatial pathogen spread.
\newblock \emph{Ecology {L}etters}, 24\penalty0 (5):\penalty0 958--969, 2021.

\bibitem[Denning(1794)]{denning1794decay}
W.~Denning.
\newblock On the decay of apple trees.
\newblock \emph{New York Society for the Promotion of Agricultural Arts and Manufacturers Transaction}, 2:\penalty0 219--222, 1794.

\bibitem[Diekmann(1978)]{diekmann1978thresholds}
O.~Diekmann.
\newblock Thresholds and travelling waves for the geographical spread of infection.
\newblock \emph{Journal of {M}athematical {B}iology}, 6\penalty0 (2):\penalty0 109--130, 1978.

\bibitem[Ducrot and Giletti(2014)]{DucrotGiletti}
A.~Ducrot and T.~Giletti.
\newblock Convergence to a pulsating travelling wave for an epidemic reaction-diffusion system with non-diffusive susceptible population.
\newblock \emph{Journal of {M}athematical {B}iology}, 69\penalty0 (3):\penalty0 533--552, 2014.
\newblock ISSN 0303-6812.
\newblock \doi{10.1007/s00285-013-0713-3}.

\bibitem[Dunbar(1983)]{dunbar1983travelling}
S.R. Dunbar.
\newblock Travelling wave solutions of diffusive lotka-volterra equations.
\newblock \emph{Journal of Mathematical Biology}, 17\penalty0 (1):\penalty0 11--32, 1983.

\bibitem[Eden-Green and Knee(1974)]{eden1974bacterial}
S.J. Eden-Green and M.~Knee.
\newblock Bacterial polysaccharide and sorbitol in fireblight exudate.
\newblock \emph{Microbiology}, 81\penalty0 (2):\penalty0 509--512, 1974.

\bibitem[El-Goorani and El-Kasheir(1989)]{el1989distribution}
M.A. El-Goorani and H.M. El-Kasheir.
\newblock Distribution of streptomycin resistant strains of {\it {e}rwinia amylovora} and control of fire blight disease in {E}gypt.
\newblock \emph{Journal of Phytopathology}, 124\penalty0 (2):\penalty0 137--142, 1989.

\bibitem[Fisher(1937)]{fisher1937wave}
R.A. Fisher.
\newblock The wave of advance of advantageous genes.
\newblock \emph{Annals of {E}ugenics}, 7\penalty0 (4):\penalty0 355--369, 1937.

\bibitem[Giletti(2010)]{giletti2010}
Thomas Giletti.
\newblock {KPP} reaction-diffusion system with a nonlinear loss inside a cylinder.
\newblock \emph{Nonlinearity}, 23\penalty0 (9):\penalty0 2307--2332, 2010.
\newblock ISSN 0951-7715.
\newblock \doi{10.1088/0951-7715/23/9/012}.

\bibitem[Heinrich(1979)]{heinrich1979resource}
B.~Heinrich.
\newblock Resource heterogeneity and patterns of movement in foraging bumblebees.
\newblock \emph{Oecologia}, 40\penalty0 (3):\penalty0 235--245, 1979.

\bibitem[Hildebrand(1936)]{hildebrand1936overwintering}
E.M. Hildebrand.
\newblock Overwintering of {\it {e}rwinia amylovora} in association with severe winter injury on baldwin apple trees.
\newblock \emph{Phytopathology}, 26:\penalty0 702--707, 1936.

\bibitem[Hildebrand(1939)]{hildebrand1939studies}
E.M. Hildebrand.
\newblock Studies on fire-blight ooze.
\newblock \emph{Phytopathology}, 29\penalty0 (2):\penalty0 142--156, 1939.

\bibitem[Hildebrand and Phillips(1936)]{hildebrand1936honeybee}
E.M. Hildebrand and E.F. Phillips.
\newblock The honeybee and the beehive in relation to fire blight.
\newblock \emph{Journal of {A}gricultural {R}esearch}, 52:\penalty0 789--810, 1936.

\bibitem[Hosono and Ilyas(1995)]{Hosono1995}
Yuzo Hosono and Bilal Ilyas.
\newblock Traveling waves for a simple diffusive epidemic model.
\newblock \emph{Math. Models Methods Appl. Sci.}, 5\penalty0 (7):\penalty0 935--966, 1995.
\newblock ISSN 0218-2025.
\newblock \doi{10.1142/S0218202595000504}.

\bibitem[Huang and Zou(2006)]{huang2006travelling}
J.~Huang and X.~Zou.
\newblock Travelling wave solutions in delayed reaction diffusion systems with partial monotonicity.
\newblock \emph{Acta Mathematicae Applicatae Sinica}, 22\penalty0 (2):\penalty0 243--256, 2006.

\bibitem[Iljon et~al.(2012)Iljon, Stirling, and Smith?]{iljon2012mathematical}
T.~Iljon, J.~Stirling, and R.J. Smith?
\newblock A mathematical model describing an outbreak of fire blight.
\newblock \emph{Understanding the Dynamics of Emerging and Re-Emerging Infectious Diseases Using Mathematical Models}, 2012.

\bibitem[Jones et~al.(2000)Jones, Schnabel, et~al.]{jones2000development}
A.L Jones, E.L. Schnabel, et~al.
\newblock The development of streptomycin-resistant strains of {\it {e}rwinia amylovora}.
\newblock \emph{Fire blight: the disease and its causative agent, {\it {E}rwinia amylovora}}, pages 235--251, 2000.

\bibitem[K{\"a}ll{\'e}n(1984)]{Kallen1984}
Anders K{\"a}ll{\'e}n.
\newblock Thresholds and travelling waves in an epidemic model for rabies.
\newblock \emph{Nonlinear Anal., Theory Methods Appl.}, 8:\penalty0 851--856, 1984.
\newblock ISSN 0362-546X.
\newblock \doi{10.1016/0362-546X(84)90107-X}.

\bibitem[Kermack and McKendrick(1927)]{kermack1927contribution}
W.O. Kermack and A.G. McKendrick.
\newblock A contribution to the mathematical theory of epidemics.
\newblock \emph{Proceedings of the Royal Society of London. Series A}, 115\penalty0 (772):\penalty0 700--721, 1927.

\bibitem[Kolmogorov(1937)]{kolmogorov1937study}
A.N. Kolmogorov.
\newblock A study of the equation of diffusion with increase in the quantity of matter, and its application to a biological problem.
\newblock \emph{Moscow University Bulletin of Mathematics}, 1:\penalty0 1--25, 1937.

\bibitem[Lafferty(2004)]{lafferty2004fishing}
K.D. Lafferty.
\newblock Fishing for lobsters indirectly increases epidemics in sea urchins.
\newblock \emph{Ecological {A}pplications}, 14\penalty0 (5):\penalty0 1566--1573, 2004.

\bibitem[Li(2012)]{li2012traveling}
B.~Li.
\newblock Traveling wave solutions in partially degenerate cooperative reaction--diffusion systems.
\newblock \emph{Journal of Differential Equations}, 252\penalty0 (9):\penalty0 4842--4861, 2012.

\bibitem[Logan(2001)]{Logan2001hydrogeo}
J.D. Logan.
\newblock \emph{Transport Modeling in Hydrogeochemical Systems}.
\newblock Springer, 2001.

\bibitem[Logan(2008)]{logan2008introduction}
J.D. Logan.
\newblock \emph{An {I}ntroduction to {N}onlinear {P}artial {D}ifferential {E}quations}.
\newblock John Wiley \& Sons, 2008.

\bibitem[Ma(2001)]{ma2001traveling}
S.~Ma.
\newblock Traveling wavefronts for delayed reaction-diffusion systems via a fixed point theorem.
\newblock \emph{Journal of Differential Equations}, 171\penalty0 (2):\penalty0 294--314, 2001.

\bibitem[Madden et~al.(2007)Madden, Hughes, and Van Den~Bosch]{madden2007study}
L.V. Madden, G.~Hughes, and F.~Van Den~Bosch.
\newblock \emph{The {S}tudy of {P}lant {D}isease {E}pidemics}.
\newblock American Phytopathological Society, 2007.

\bibitem[Miller(1929)]{miller1929studies}
P.W. Miller.
\newblock Studies of fire blight of apple in wisconsin.
\newblock \emph{Journal of {A}gricultural {R}esearch}, 39:\penalty0 579--621, 1929.

\bibitem[Mills(1955)]{mills1955fire}
W.D. Mills.
\newblock Fire blight development on apple in {W}estern {N}ew {Y}ork.
\newblock \emph{Plant Disease Reporter}, 39\penalty0 (3):\penalty0 206--207, 1955.

\bibitem[Minogue and Fry(1983)]{minogue1983models}
KP~Minogue and WE~Fry.
\newblock Models for the spread of plant disease: Some experimental results.
\newblock \emph{Phytopathology}, 73\penalty0 (8):\penalty0 1173--1176, 1983.

\bibitem[Okubo and Levin(2001)]{okubo2001diffusion}
A.~Okubo and S.A. Levin.
\newblock \emph{Diffusion and {E}cological {P}roblems: {M}odern {P}erspectives}, volume~14.
\newblock Springer, 2001.

\bibitem[{R Core Team}(2021)]{Rcite}
{R Core Team}.
\newblock \emph{R: A Language and Environment for Statistical Computing}.
\newblock R Foundation for Statistical Computing, Vienna, Austria, 2021.
\newblock URL \url{https://www.R-project.org/}.

\bibitem[Rosen(1929)]{rosen1929life}
H.R. Rosen.
\newblock Life history of the fire blight pathogen, {\it {b}acillus amylovorus}, as related to the means of over wintering and dissemination.
\newblock \emph{Bulletin of the {A}rkansas, {A}gricultural {E}xperiment {S}tation}, 244, 1929.

\bibitem[Rosen(1933)]{rosen1933further}
H.R. Rosen.
\newblock Further studies on the overwintering and dissemination of the fire-blight pathogen.
\newblock \emph{Arkansas {A}gricultural {E}xperiment {S}tation {B}ulletin}, 283, 1933.

\bibitem[Rosen(1936)]{rosen1936mode}
H.R. Rosen.
\newblock Mode of penetration and of progressive invasion of fire-blight bacteria into apple and pear blossoms.
\newblock \emph{Bulletin of the {A}rkansas {A}gricultural {E}xperiment {S}tation}, 331, 1936.

\bibitem[Sanchirico and Wilen(2005)]{sanchirico2005optimal}
J.N. Sanchirico and J.E. Wilen.
\newblock Optimal spatial management of renewable resources: matching policy scope to ecosystem scale.
\newblock \emph{Journal of Environmental Economics and Management}, 50\penalty0 (1):\penalty0 23--46, 2005.

\bibitem[Shang et~al.(2016)Shang, Du, and Lin]{shang2016traveling}
X.~Shang, Z.~Du, and X.~Lin.
\newblock Traveling wave solutions of n-dimensional delayed reaction--diffusion systems and application to four-dimensional predator--prey systems.
\newblock \emph{Mathematical Methods in the Applied Sciences}, 39\penalty0 (6):\penalty0 1607--1620, 2016.

\bibitem[Shu et~al.(2019)Shu, Pan, Wang, and Wu]{shu2019traveling}
H.~Shu, X.~Pan, X.~Wang, and J.~Wu.
\newblock Traveling waves in epidemic models: non-monotone diffusive systems with non-monotone incidence rates.
\newblock \emph{Journal of Dynamics and Differential Equations}, 31\penalty0 (2):\penalty0 883--901, 2019.

\bibitem[Slack et~al.(2022)Slack, Schachterle, Sweeney, Kharadi, Peng, Botti-Marino, Bardaji, Pochubay, and Sundin]{slack2022orchard}
S.~Slack, J.~Schachterle, E.~Sweeney, R.~Kharadi, J.~Peng, M.~Botti-Marino, L.~Bardaji, E.~Pochubay, and G.W. Sundin.
\newblock In-orchard population dynamics of erwinia amylovora on apple flower stigmas.
\newblock \emph{Phytopathology}, 112\penalty0 (6):\penalty0 1214--1225, 2022.

\bibitem[Slack and Sundin(2017)]{slack2017news}
S.M. Slack and G.W. Sundin.
\newblock News on ooze, the fire blight spreader.
\newblock \emph{Fruit Quarterly}, 25\penalty0 (1):\penalty0 9--12, 2017.

\bibitem[Slack et~al.(2017)Slack, Zeng, Outwater, and Sundin]{slack2017microbiological}
S.M. Slack, Q.~Zeng, C.A. Outwater, and G.W. Sundin.
\newblock Microbiological examination of {\it {e}rwinia amylovora} exopolysaccharide ooze.
\newblock \emph{Phytopathology}, 107\penalty0 (4):\penalty0 403--411, 2017.

\bibitem[Smith(1992)]{smith1992predictive}
T.J. Smith.
\newblock A predictive model for forecasting fire blight of pear and apple in washington state.
\newblock In \emph{VI International Workshop on Fire Blight 338}, pages 153--160, 1992.

\bibitem[Smoller(2012)]{smoller2012shock}
J.~Smoller.
\newblock \emph{Shock waves and reaction—diffusion equations}, volume 258.
\newblock Springer Science \& Business Media, 2012.

\bibitem[Soetaert et~al.(2010)Soetaert, Petzoldt, and Setzer]{soetaert2010solving}
K.~Soetaert, T.~Petzoldt, and R.W. Setzer.
\newblock Solving differential equations in r: package desolve.
\newblock \emph{Journal of {S}tatistical {S}oftware}, 33:\penalty0 1--25, 2010.

\bibitem[Steiner(1989)]{steiner1989predicting}
P.W. Steiner.
\newblock Predicting apple blossom infections by {\it {}rwinia amylovora} using the maryblyt model.
\newblock In \emph{V International Workshop on Fire Blight 273}, pages 139--148, 1989.

\bibitem[Thomson and Gouk(2003)]{thomson2003influence}
S.V. Thomson and S.C. Gouk.
\newblock Influence of age of apple flowers on growth of {\it {e}rwinia amylovora} and biological control agents.
\newblock \emph{Plant Disease}, 87\penalty0 (5):\penalty0 502--509, 2003.

\bibitem[Thomson et~al.(1992)Thomson, Gouk, Vanneste, Hale, and Clark]{thomson1992presence}
S.V. Thomson, S.C. Gouk, J.L. Vanneste, C.N. Hale, and R.G. Clark.
\newblock The presence of streptomycin resistant strains of erwinia amylovora in new zealand.
\newblock In \emph{VI International Workshop on Fire Blight 338}, pages 223--230, 1992.

\bibitem[Thomson et~al.(1986)]{thomson1986role}
S.V. Thomson et~al.
\newblock The role of the stigma in fire blight infections.
\newblock \emph{Phytopathology}, 76\penalty0 (5):\penalty0 476--482, 1986.

\bibitem[Van~den Bosch et~al.(1977)Van~den Bosch, Zadoks, and Metz]{van1977focus}
F~Van~den Bosch, JC~Zadoks, and JAJ Metz.
\newblock Focus expansion in plant disease. i: The constant rate of focus expansion.
\newblock \emph{Annuals of the {N}ew {Y}ork {A}cademy of {S}cience}, 287\penalty0 (171), 1977.

\bibitem[Van~den Bosch et~al.(1988)Van~den Bosch, Frinking, Metz, and Zadoks]{van1988focus}
F.~Van~den Bosch, H.D. Frinking, J.A.J. Metz, and J.C. Zadoks.
\newblock Focus expansion in plant disease. iii: Two experimental examples.
\newblock \emph{Phytopathology}, 78\penalty0 (7):\penalty0 919--925, 1988.

\bibitem[Van~der Zwet et~al.(2012)Van~der Zwet, Orolaza-Halbrendt, and Zeller]{van2012losses}
T.~Van~der Zwet, N.~Orolaza-Halbrendt, and W.~Zeller.
\newblock \emph{Fire {B}light: {H}istory, {B}iology,and {M}anagement}.
\newblock American Phytopathological Society, 2012.

\bibitem[Van~Laere et~al.(1980)Van~Laere, De~Greef, and De~Wael]{van1980influence}
O.~Van~Laere, M.~De~Greef, and L.~De~Wael.
\newblock Influence of the honeybee on fireblight transmission.
\newblock In \emph{ISHS {A}cta {H}orticulturae 117: {II} {S}ymposium on {F}ire {B}light}, pages 131--144, 1980.

\bibitem[Viswanathan et~al.(1999)Viswanathan, Buldyrev, Havlin, Da~Luz, Raposo, and Stanley]{viswanathan1999optimizing}
G.M. Viswanathan, S.V. Buldyrev, S.~Havlin, M.G.E. Da~Luz, E.P. Raposo, and H.E. Stanley.
\newblock Optimizing the success of random searches.
\newblock \emph{Nature}, 401\penalty0 (6756):\penalty0 911--914, 1999.

\bibitem[Volpert et~al.(1994)Volpert, Volpert, and Volpert]{volpert1994traveling}
A.I. Volpert, V.A. Volpert, and V.A. Volpert.
\newblock \emph{Traveling wave solutions of parabolic systems}, volume 140.
\newblock American Mathematical Soc., 1994.

\bibitem[Waite(1891)]{waite1891results}
M.B. Waite.
\newblock Results from recent investigations in pear blight.
\newblock \emph{Bot. Gaz}, 16:\penalty0 259, 1891.

\bibitem[Wang and Wang(2016)]{wang2016traveling}
H.~Wang and X.~Wang.
\newblock Traveling wave phenomena in a {K}ermack--{M}c{K}endrick {SIR} model.
\newblock \emph{Journal of {D}ynamics and {D}ifferential {E}quations}, 28\penalty0 (1):\penalty0 143--166, 2016.

\bibitem[Wang et~al.(2012)Wang, Wang, and Wu]{wang2012traveling}
X.~Wang, H.~Wang, and J.~Wu.
\newblock Traveling waves of diffusive predator-prey systems: disease outbreak propagation.
\newblock \emph{Discrete \& Continuous Dynamical Systems}, 32\penalty0 (9):\penalty0 3303, 2012.

\bibitem[Wilson et~al.(1989)Wilson, Sigee, and Epton]{wilson1989erwinia}
M.~Wilson, D.C. Sigee, and H.A.S. Epton.
\newblock {\it Erwinia amylovora} infection of {H}awthorn blossom: {I}. the anther.
\newblock \emph{Journal of {P}hytopathology}, 127\penalty0 (1):\penalty0 1--14, 1989.

\bibitem[Wilson et~al.(1990)Wilson, Sigee, and Epton]{wilson1990erwinia}
M.~Wilson, D.C. Sigee, and H.A.S. Epton.
\newblock {\it Erwinia amylovra} infection of {H}awthorn blossom: {III}. the nectary.
\newblock \emph{Journal of {P}hytopathology}, 128\penalty0 (1):\penalty0 62--74, 1990.

\bibitem[Zadoks and Van~den Bosch(1994)]{zadoks1994spread}
J.C. Zadoks and F.~Van~den Bosch.
\newblock On the spread of plant disease: a theory on foci.
\newblock \emph{Annual {R}eview of {P}hytopathology}, 32\penalty0 (1):\penalty0 503--521, 1994.

\bibitem[Zhou et~al.(2019)Zhou, Song, Wei, and Xu]{zhou2019critical}
J.~Zhou, L.~Song, J.~Wei, and H.~Xu.
\newblock Critical traveling waves in a diffusive disease model.
\newblock \emph{Journal of Mathematical Analysis and Applications}, 476\penalty0 (2):\penalty0 522--538, 2019.

\end{thebibliography}
\bibliographystyle{plainnat}
\end{document}